\documentclass[12pt,english]{amsart}
\usepackage[T1]{fontenc}
\usepackage[latin9]{inputenc}
\usepackage{geometry}
\usepackage[active]{srcltx}
\usepackage{prettyref}
\usepackage{amsthm}
\usepackage{amstext}
\usepackage{amssymb}
\usepackage{esint}

\makeatletter
\numberwithin{equation}{section}
\numberwithin{figure}{section}
\usepackage{enumitem}		
\theoremstyle{plain}
\newtheorem{thm}{\protect\theoremname}[section]
\theoremstyle{plain}
\newtheorem{prop}[thm]{\protect\propositionname}
\ifx\proof\undefined
\newenvironment{proof}[1][\protect\proofname]{\par
\normalfont\topsep6\p@\@plus6\p@\relax
\trivlist
\itemindent\parindent
\item[\hskip\labelsep
\scshape
#1]\ignorespaces
}{%
\endtrivlist\@endpefalse
}
\providecommand{\proofname}{Proof}
\fi
\theoremstyle{plain}
\newtheorem{lem}[thm]{\protect\lemmaname}

\makeatletter
\@addtoreset{thm}{section}
\makeatother
\usepackage{prettyref}
\usepackage{refstyle}

\makeatother

\usepackage{babel}
\providecommand{\lemmaname}{Lemma}
\providecommand{\propositionname}{Proposition}
\providecommand{\theoremname}{Theorem}

\begin{document}

\title{Asymptotics of the eta invariant}

\author{nikhil savale}
\begin{abstract}
We prove an asymptotic bound on the eta invariant of a family of coupled
Dirac operators on an odd dimensional manifold. In the case when the
manifold is the unit circle bundle of a positive line bundle over
a complex manifold, we obtain precise formulas for the eta invariant.
\end{abstract}

\address{Université Paris-Sud, Département de Mathématique, F-91405 Orsay
FRANCE }

\email{nikhil.savale@math.u-psud.fr }

\thanks{The research leading to the results contained in this paper has received
funding from the European Research Council (E.R.C.) under European
Union's Seventh Framework Program (FP7/2007-2013)/ ERC grant agreement
No. 291060.}

\maketitle
\global\long\def\norm#1{\lVert#1\rVert}

\section{Introduction}

The eta invariant was introduced by Atiyah, Patodi and Singer in \cite{APSI}
as a correction term to an index theorem for manifolds with boundary.
Consider a first order, elliptic and self-adjoint operator $A$ on
a compact manifold. Formally, the eta invariant $\eta(A)$ of this
operator can be interpreted as its signature, or the difference between
the number of positive and the number of negative eigenvalues of $A$.
In reality, since $A$ has infinitely many eigenvalues of each sign
this needs to be defined via regularization (see \prettyref{sec:Preliminaries}).

A key feature of the invariant $\eta(A)$, much like the signature
of a matrix, is that it is $\mathit{not}$ in general a continuous
function of the operator $A$. In particular consider a smooth one-parameter
family of operators $A_{t}$. The corresponding eta invariant $\eta(A_{t})$
is then in general a discontinuous function of the parameter $t$,
making it difficult to understand how it behaves as $t$ varies. In
this paper we shall investigate how the eta invariant of such a one
parameter family behaves asymptotically as the parameter gets large.

More precisely, consider a compact, oriented Riemannian manifold $(Y,g^{TY})$
of odd dimension $n=2m+1$, equipped with a spin structure. Let $S$
be the corresponding spin bundle on $Y$. Let $L$ be a Hermitian
line bundle on $Y$. Let $A_{0}$ be a fixed unitary connection on
$L$ and let $a\in\Omega^{1}(Y;i\mathbb{R})$ be an imaginary one
form on $Y$. This gives a family $A_{r}=A_{0}+ra$ of unitary connections
on $L$, with $r\in\mathbb{R}$ being a real parameter . Each connection
in this family gives a coupled Dirac operator $D_{A_{r}}$ acting
on sections of $S\otimes L$. Our first result, regarding the asymptotics
of the reduced eta invariant $\bar{\eta}^{r}=\bar{\eta}(D_{A_{r}})$,
is the following.
\begin{thm}
\label{thm:main theorem}As $r\rightarrow\infty$, the reduced eta
invariant satisfies the asymptotics 
\begin{eqnarray}
\bar{\eta}^{r} & = & o(r^{\frac{n}{2}}).\label{eq:main estimate little o}
\end{eqnarray}

\end{thm}
It is an interesting question as to what extent the little $o(r^{\frac{n}{2}})$
estimate of \prettyref{thm:main theorem} can be improved. 

In order to investigate this question we consider the eta invariant
of such a family in the case where $Y$ is the total space of a circle
bundle. In particular, we shall let $Y$ be the space of unit elements
of a positive line bundle $\mathcal{L}\rightarrow X$ over a complex
manifold $X$. We shall further equip $Y$ with an adiabatic family
of metrics $g_{\varepsilon}^{TY}$ (see \prettyref{sec:Eta-invariant-of}).
Under an appropriate choice of the family of connections, this gives
the corresponding eta invariant $\bar{\eta}^{r,\varepsilon}$, with
now an additional dependence on the adiabatic parameter $\varepsilon$.
Letting $\hat{A}(X)$ denote the $\hat{A}$-genus of $X$, we now
prove the following more precise formula for the eta invariant (see
theorem \prettyref{thm:Eta invrian asmp S1 bundle})
\begin{thm}
The eta invariant $\bar{\eta}^{r,\varepsilon}$ satisfies the asymptotics
\begin{align*}
\bar{\eta}^{r,\varepsilon}= & \sum_{a=0}^{m}\left\{ \left(\frac{r^{a+1}}{(a+1)!}-\sum_{k=1}^{\left[r+\frac{\varepsilon m}{2}\right]}\frac{k^{a}}{a!}\right)\int_{X}c_{1}(\mathcal{L})^{a}\left[\hat{A}(X)\right]^{m-a}\right\} +O(1).
\end{align*}
as $r\rightarrow\infty$.
\end{thm}
From this formula we observe that $\bar{\eta}^{r,\varepsilon}$, in
this case, exhibits jump discontinuities at integer values of $r+\frac{\varepsilon m}{2}$.
Furthermore, the size of the jumps is growing at the rate $r^{\frac{n-1}{2}}$
as $r\rightarrow\infty$ . Hence this calculation demonstrates that
\prettyref{thm:main theorem} cannot be improved beyond an $O(r^{\frac{n-1}{2}})$
estimate on the eta invariant.

The eta invariant is a non-local quantity. That is, it cannot be written
as an integral over the manifold of a canonical differential from
obtained from the symbol of the operator. This makes it difficult
to compute the eta invariant explicitly. In the final section of this
paper we give an exact formula for the eta invariant $\bar{\eta}^{r,\varepsilon}$,
assuming the value of the adiabatic parameter $\varepsilon$ to be
small, using the adiabatic limit technique of Bismut-Cheeger, Dai
and Zhang \cite{Bismut-Cheeger,Dai,Zhang}. We refer to \prettyref{thm:Eta invariant explicit computation}
for the exact formula arising from the computation. A striking feature
of this formula is that it expresses the eta invariant $\bar{\eta}^{r,\varepsilon}$
in purely topological terms on the base $X$. This generalizes a similar
known computation in dimension three of Nicolaescu \cite{Nicolaescu-Eta}.

An asymptotic result of the form \prettyref{thm:main theorem} was
used by Taubes in \cite{Taubes-SpectralFlow,Taubes-Weinstein} in
order to prove the Weinstein conjecture on the existence of Reeb orbits
on three dimensional contact manifolds. Our results improve the estimates
obtained therein and could lead to further information regarding Reeb
orbits. The three dimensional case has been further explored, under
certain hypotheses, by Tsai in \cite{Tsai-thesis}.

In another direction, the asymptotics considered in this paper are
closely related to the asymptotic results of Bismut-Vasserot from
\cite{Bismut,Bismut-Vasserot}. In \cite{Bismut-Vasserot} the authors
considered the Dolbeault Laplacian $\Delta_{\bar{\partial_{k}}}^{p}$
acting on $p$-forms, with values in a tensor power $\mathcal{L}^{\otimes k}$,
of the positive line bundle $\mathcal{L}$ considered here earlier.
They then derived an asymptotic formula for the holomorphic torsion
of $\Delta_{\bar{\partial_{k}}}^{p}$ in the limit as $k\rightarrow\infty$.
The asymptotics of the heat trace of $\Delta_{\bar{\partial_{k}}}^{p}$
were used in \cite{Bismut} to prove Demailly's asymptotic Morse inequalities.
This Laplacian will arise in our computations in \prettyref{sec:Eta-invariant-of}
and it would be interesting to explore this connection further. 

The paper is organized as follows. In \prettyref{sec:Preliminaries}
we begin with preliminary notations and facts used in the paper. In
\prettyref{sec:Asymptotics-of-the} we derive asymptotics of heat
traces required in the proof of \prettyref{thm:main theorem}. In
\prettyref{sec:Asymptotics-} we derive the asymptotics of the spectral
measure of a rescaled Dirac operator and prove \prettyref{thm:main theorem}.
In \prettyref{sec:Eta-invariant-of} we consider the eta invariant
of the circle bundle. There we prove \prettyref{thm:Eta invrian asmp S1 bundle}
and give the exact computation for the eta invariant of \prettyref{thm:Eta invariant explicit computation}.

\section{Preliminaries\label{sec:Preliminaries}}

\noindent Consider a compact, oriented, Riemannian manifold $(Y,g^{TY})$
of odd dimension $n$ equipped with a spin structure. Let $S$ be
the corresponding spin bundle on $Y$. Let $\nabla^{TY}$ denote the
Levi-Civita connection on $TY$. This lifts to the spin connection
$\nabla^{S}$ on the spin bundle $S$. We denote the Clifford multiplication
endomorphism by $c:T^{*}Y\rightarrow S\otimes S^{*}$ satisfying
\begin{align*}
c(a)^{2}=-|a|^{2}, & \quad\forall a\in T^{*}Y.
\end{align*}
Let $L$ be a Hermitian line bundle on $Y$. Let $A_{0}$ be a fixed
unitary connection on $L$ and let $a\in\Omega^{1}(Y;i\mathbb{R})$
be an imaginary 1-form on $Y$. This gives a family $A_{r}=A_{0}+ra$
of unitary connections on $L$. We denote by $\nabla^{r}=\nabla^{S}\otimes1+1\otimes A_{r}$
the tensor product connection on $S\otimes L.$ Each such connection
defines a coupled Dirac operator 
\begin{align*}
D_{A_{r}}=c\circ\nabla^{r}:C^{\infty}(Y;S\otimes L)\rightarrow C^{\infty}(Y;S\otimes L).
\end{align*}
Each Dirac operator $D_{A_{r}}$ is elliptic and self-adjoint. It
hence possesses a discrete spectrum of eigenvalues. Define the eta
function of $D_{A_{r}}$ by the formula
\begin{align}
\eta(D_{A_{r}},s)= & \sum_{\begin{subarray}{l}
\quad\:\lambda\neq0\\
\lambda\in\textrm{Spec}(D_{A_{r}})
\end{subarray}}\textrm{sign}(\lambda)|\lambda|^{-s}=\frac{1}{\Gamma(\frac{s+1}{2})}\int_{0}^{\infty}t^{\frac{s-1}{2}}\textrm{tr}\left(D_{A_{r}}e^{-tD_{A_{r}}^{2}}\right)dt.\label{eq:eta invariant definition}
\end{align}
Here, and in the remainder of the paper, we use the convention that
$\textrm{Spec}(D_{A_{r}})$ denotes a multiset with each eigenvalue
of $D_{r}$ being counted with its multiplicity. The above series
converges for $\textrm{Re}(s)>n.$ It was shown in \cite{APSI,APSIII}
that the eta function possesses a meromorphic continuation to the
entire complex $s$-plane and has no pole at zero. Its value at zero
is defined to be the eta invariant of the operator $\eta(D_{A_{r}})=\eta(D_{A_{r}},0)$.
By including the zero eigenvalue in \prettyref{eq:eta invariant definition},
with an appropriate convention, we may define a variant known as the
reduced eta invariant by 
\begin{align*}
\bar{\eta}(D_{A_{r}})= & \frac{1}{2}\left\{ \textrm{dim ker}(D_{A_{r}})+\eta(D_{A_{r}})\right\} .
\end{align*}
We shall henceforth denote the reduced eta invariant by the shorthand
$\bar{\eta}^{r}=\bar{\eta}(D_{A_{r}})$, and would like to investigate
its asymptotics for large $r$. Our results will apply equally well
to the unreduced version $\eta^{r}=\eta(D_{A_{r}})$. 

Let $L_{t}^{r}$ denote the Schwartz kernel of the operator $D_{A_{r}}e^{-tD_{A_{r}}^{2}}$
on the product $Y\times Y$. Denote by $\textrm{tr}\left(L_{t}^{r}\left(x,x\right)\right)$
the pointwise trace of $L_{t}^{r}$ along the diagonal. We may now
analogously define the function 
\begin{align}
\eta(D_{A_{r}},s,x)= & \frac{1}{\Gamma(\frac{s+1}{2})}\int_{0}^{\infty}t^{\frac{s-1}{2}}\textrm{tr}\left(L_{t}^{r}\left(x,x\right)\right)dt.\label{eq:eta function diagonal}
\end{align}
In \cite{Bismut-Freed-II} theorem 2.6, the authors showed that for
$\textrm{Re}(s)>-2$, the function $\eta(D_{A_{r}},s,x)$ is holomorphic
in $s$ and smooth in $x$. From \prettyref{eq:eta function diagonal}
it is clear that this is equivalent to 
\begin{align}
\textrm{tr}(L_{t}^{r})= & O(t^{\frac{1}{2}}),\quad\textrm{as}\: t\rightarrow0.\label{eq:pointwise trace asymp as t->0}
\end{align}

\section{Asymptotics of the heat kernel\label{sec:Asymptotics-of-the}}

In order to control the eta invariant we shall need to find the asymptotics
for the heat traces of $D_{A_{r}}$. We begin with an estimate on
its heat kernel. We denote by $dy$ the Riemannian volume form on
$(Y,g)$. All kernels will be calculated with respect to $dy$ in
what follows. Let $i_{g}$ denote the injectivity radius of $Y.$
Let $\rho(x,y)$ denote the geodesic distance function between two
given points $x,y\in Y$. Define a function on $Y\times Y$ by the
following formula 
\[
h_{t}(x,y)=\frac{e^{-\frac{\rho(x,y)^{2}}{4t}}}{(4\pi t)^{\frac{n}{2}}}.
\]
Let $H_{t}^{r}(x,y)$ denote the kernel of $e^{-tD_{A_{r}}^{2}}$
for $t>0$. We now have the following estimate. 
\begin{prop}
\label{prop:Maximum Principle estimate}There exist positive constants
$c_{1},c_{2}$ independent of $r$ such that 
\begin{equation}
|H_{t}^{r}(x,y)|\leq c_{1}h_{2t}(x,y)e^{c_{2}rt}\label{eq:mest}
\end{equation}
for all $x,y\in Y,$$t>0$ and $r\geq1$.\end{prop}
\begin{proof}
Let $\nabla^{S}$ denote the spin connection on $S$ and $\nabla^{r}=\nabla^{S}\otimes1+1\otimes A_{r}$
be the tensor product connection on $S\otimes L.$ First observe that
for fixed $y$ the section $s_{t}(.)=H_{t}^{r}(.,y)$ satisfies the
heat equation $\partial_{t}s_{t}=-D_{A_{r}}^{2}s_{t}$. The Weitzenbock
formula gives

\[
D_{A_{r}}^{2}=\nabla^{r*}\nabla^{r}+c(F_{A_{0}})+rc(da)+\frac{\kappa}{4}
\]
where $F_{A_{0}},\kappa$ denote the curvature of $A_{0}$ and the
scalar curvature of $g$ respectively. Using the Weitzenbock formula
and the heat equation $\partial_{t}s_{t}=-D_{A_{r}}^{2}s_{t}$, we
now see that the function $f_{t}=|s_{t}|$ obeys the inequality

\begin{equation}
\partial_{t}f_{t}\leq-d^{*}df_{t}+c_{1}(r+1)f_{t}
\end{equation}
 for some constant $c_{1}>0$ independent of $r$. Hence the function
$f{}_{t}^{0}=e^{-c_{1}(r+1)t}f_{t}$ satisfies the inequality 
\begin{equation}
\partial_{t}f{}_{t}^{0}\leq-d^{*}df{}_{t}^{0}.\label{eq:subsol}
\end{equation}
Let $\Phi_{t}(x,y)$ denote the heat kernel $e^{-td^{*}d}$ for the
Laplace operator acting on functions on $Y$. Now since $|H_{t}^{r}(x,y)|$
and $\Phi_{t}(x,y)$ have the same asymptotics as $t\rightarrow0$,
an application of the maximum principle for the heat equation gives 

\begin{equation}
f{}_{t}^{0}\leq\Phi_{t}(x,y)\label{eq:scalar heat kernel comp.}
\end{equation}
 for all time $t>0$. Next we use the estimate 
\begin{equation}
\Phi_{t}(x,y)\leq c_{3}e^{t}h_{2t}(x,y),\quad\forall t>0,\label{eq:scalar heat kernel estimate}
\end{equation}
 on the heat kernel. Equation \prettyref{eq:scalar heat kernel estimate}
follows for large time since the heat kernel is bounded 
\begin{align*}
\Phi_{t}(x,y)\leq c_{4} & \quad\textrm{for }\quad\forall x,y\in Y\,\textrm{and }\, t\geq1.
\end{align*}
For small time, \prettyref{eq:scalar heat kernel estimate} follows
from the heat kernel estimate of \cite{Cheng-Li-Yau}. The proposition
now follows from \prettyref{eq:scalar heat kernel comp.} and \prettyref{eq:scalar heat kernel estimate}.
\end{proof}
Following this we shall prove a more refined estimate on the heat
kernel comparing it with Mehler's kernel. We first recall the definition
of the Mehler's kernel. Define an antisymmetric endomorphism $A$
of $TY$ via 
\begin{equation}
ida(X,Y)=g(X,AY),\,\forall X,Y\in TY.\label{eq:antisymm endomorphism A}
\end{equation}
Let $x,y$ be two points of $Y$ such that $\rho(x,y)<i_{g}$. Let
$v\in T_{y}Y$ such that $x=\exp_{v}y$. Define a function on a geodesic
neighborhood of the diagonal in $Y\times Y$ by
\begin{equation}
m_{t}^{r}(x,y)=\frac{1}{(4\pi t)^{\frac{n}{2}}}\det{}^{\frac{1}{2}}\left(\frac{rtA_{y}}{\sinh rtA_{y}}\right)\exp\left\{ -\frac{1}{4t}g(v,rtA_{y}\coth(rtA_{y})v)\right\} \label{eq:function m_t in Mehler's kernel}
\end{equation}
Now let $\pi_{1}$ and $\pi_{2}$ denote the projections onto the
two factors of $Y\times Y$ and define a section $e^{-tF_{A_{r}}}$
of $\pi_{1}^{*}(S\otimes L)\otimes\pi_{2}^{*}(S\otimes L)^{*}$, in
a geodesic neighborhood of the diagonal. This restricts to $e^{-tF_{A_{r}}}|_{\Delta}=e^{-tc(F_{A_{r}})}$
at the diagonal $\Delta$ and is parallel along geodesics $(\exp_{tv}(y),y)$.
Consider a smooth cutoff function satisfying
\begin{equation}
\chi(x)=\begin{cases}
0 & \textrm{if }\quad|x|\geq i_{g}\\
1 & \textrm{if }\quad|x|\leq\frac{i_{g}}{2}.
\end{cases}\label{eq:cutoff}
\end{equation}
Mehler's kernel is now defined via 
\begin{equation}
M_{t}^{r}(x,y)=\chi(\rho(x,y))m_{t}^{r}(x,y)e^{-tF_{A_{r}}}.\label{eq:Mehler's kernel def}
\end{equation}

\begin{prop}
\label{prop:Heat Mehler comparison}There exist positive constants
$c_{1}$ and $c_{2}$ independent of $r$ such that 
\begin{equation}
|H_{t}^{r}(x,y)-M_{t}^{r}(x,y)|\leq c_{1}h_{8t}(x,y)t^{\frac{1}{2}}e^{c_{2}rt},\label{eq:mhest}
\end{equation}
 for all $x,y\in Y,$$t>0$ and $r\geq1$.\end{prop}
\begin{proof}
First fix a point $y$ and a set of geodesic coordinates centered
at $y$. Now choose a basis $s_{\alpha}$ for $S_{y}$ and a basis
$l$ for $L_{y}$. Parallel transport this basis along geodesics using
the connections $\nabla^{S},A_{r}$ to obtain trivializations $s_{\alpha}(x)$
and $l(x)$ of $S$ and $L$ respectively near $y$. Now define local
orthonormal sections of $(S\otimes L)\otimes(S\otimes L)_{y}^{*}$
via

\begin{equation}
t_{\alpha\beta}=s_{\alpha}(x)\otimes l(x)\otimes s_{\beta}^{*}\otimes l^{*}.\label{eq:parallel transport frame}
\end{equation}
 The connection $\nabla^{r}$ can be expressed in this frame and these
coordinates as

\begin{equation}
\nabla_{i}^{r}=\partial_{i}+A_{i}^{r}+\Gamma_{i},
\end{equation}
 where each $A_{i}^{r}$ is a Christoffel symbol of $A_{r}$ (or $\text{dim}(S\otimes L)_{y}$
copies of it) and each $\Gamma_{i}$ is a Christoffel symbol of the
spin connection on $S$. Since the section $l(x)$ is obtained via
parallel transport along geodesics, the connection coefficient $A_{i}^{r}$
maybe written in terms of the curvature $F_{ij}^{r}$ of $A^{r}$
via

\begin{equation}
A_{i}^{r}(x)=\int_{0}^{1}d\rho(\rho x^{j}F_{ij}^{r}(\rho x)),
\end{equation}
with the Einstein summation convention being used. The dependence
of the curvature coefficients $F_{ij}^{r}$ on the parameter $r$
is seen to be linear $F_{ij}^{r}=F_{ij}^{0}+r(da)_{ij}^{0}$ despite
the fact that they are expressed in the $r$ dependent frame $l$.
This is because a gauge transformation from an $r$ independent frame
into $l$ changes the curvature coefficient by conjugation. Since
$L$ is a line bundle this is conjugation by a function and hence
does not change the coefficient. Next, using the Taylor expansion
$(da)_{ij}=(da)_{ij}(0)+x^{k}a_{ijk}$, we see that the connection
$\nabla^{r}$ has the form

\begin{equation}
\nabla_{i}^{r}=\partial_{i}+\frac{1}{2}rx^{j}(da)_{ij}(0)+x^{j}A_{ij}^{0}+rx^{j}x^{k}A_{ijk}+\Gamma_{i}.\label{eq:cdic}
\end{equation}
 Here $A_{ij}^{0}=\int_{0}^{1}d\rho(\rho F_{ij}^{0}(\rho x))$, $A_{ijk}=\int_{0}^{1}d\rho(\rho a_{ijk}(\rho x))$
and $\Gamma_{i}$ are all independent of $r$. 

Now using Weitzenbock's formula, we note that the operator $D_{A_{r}}^{2}$
has the form

\begin{align}
D_{A_{r}}^{2} & =\mathcal{H}+E,\quad\text{with}\label{eq:dssf}\\
\mathcal{H} & =-\partial_{i}^{2}-r(da)_{ij}(0)x^{j}\partial_{i}-\frac{r^{2}}{4}x^{i}x^{j}\left(\sum_{k}(da)_{ik}(0)(da)_{jk}(0)\right)+c\left(F_{A_{r}}\right)\quad\text{and}\\
E & =P_{ijkl}x^{k}x^{l}\partial_{i}\partial_{j}+Q_{ijk}rx^{j}x^{k}\partial_{i}+R_{i}\partial_{i}+S_{ijk}r^{2}x^{i}x^{j}x^{k}+T_{i}rx^{i}+U.\label{eq:ser}
\end{align}
 Here $P,Q,R,S,T$ and $U$ are each smooth endomorphisms of $S\otimes L$
independent of $r$. Since $(\partial_{t}+D_{A_{r}}^{2})H_{t}=0$
we now have

\begin{equation}
\left(\partial_{t}+D_{A_{r}}^{2}\right)\left(H_{t}^{r}-M_{t}^{r}\right)=-\left(\partial_{t}+\mathcal{H}\right)M_{t}^{r}-EM_{t}^{r}.\label{eq:hdhm}
\end{equation}
Note that the right hand side of \prettyref{eq:hdhm} is zero for
$\rho(x,y)>i_{g}$ since $M_{t}^{r}$ is supported in a geodesic neighborhood
of the diagonal, by \prettyref{eq:Mehler's kernel def}. From the
defining equations \prettyref{eq:function m_t in Mehler's kernel}
and \prettyref{eq:Mehler's kernel def}, Mehler's kernel is given
in geodesic coordinates via 
\begin{align}
M_{t}^{r}(x,y)= & \chi(\rho(x,y))m_{t}^{r}(x,y)e^{-tF_{A_{r}}}\label{eq:Mehler's kernel in coordinates}\\
= & \chi(|x|)\frac{1}{(4\pi t)^{\frac{n}{2}}}\det{}^{\frac{1}{2}}\left(\frac{rtA_{y}}{\sinh rtA_{y}}\right)\exp\left\{ -\frac{1}{4t}\left\langle x,rtA_{y}\coth(rtA_{y})x\right\rangle \right\} e^{-tc(F_{A_{r}})}.\label{eq:Mehler's kernel coordinates}
\end{align}
We now differentiate \prettyref{eq:Mehler's kernel coordinates} using
\prettyref{eq:dssf}-\prettyref{eq:ser} to compute the right hand
side of \prettyref{eq:hdhm}. By Mehler's formula, see section 4.2
in \cite{Berline-Getzler-Vergne}, we have $(\partial_{t}+\mathcal{H})\left\{ m_{t}^{r}(x,y)e^{-tF_{A_{r}}}\right\} =0$
for $d(x,y)<\frac{i_{g}}{2}$. Differentiating the rest, we observe
that the right hand side of \prettyref{eq:hdhm} has the form of a
finite sum

\begin{equation}
-\left(\partial_{t}+\mathcal{H}\right)M_{t}^{r}-EM_{t}^{r}=\sum_{(k,d,I)}t^{k}r^{d}x^{I}P_{k,d,I}(x)g_{k,d,I}(rt)M_{t}^{r},\;\textrm{where}\label{eq:hdsf}
\end{equation}

\begin{itemize}
\item each $(k,d,I)\in\mathbb{Z}\times\mathbb{N}_{0}\times\mathbb{N}_{0}^{n}$
and satisfies the inequality
\begin{align}
d\leq & k+\frac{|I|}{2}+\frac{1}{2},\label{eq:iin}
\end{align}

\item each $P_{k,d,I}$ appearing in \prettyref{eq:hdsf} denotes a smooth
endomorphism of $S\otimes L$, independent of $r$, and supported
on $\rho(x,y)<i_{g}$
\item and $g_{k,d,I}$ in \prettyref{eq:hdsf} denote functions, coming
from the matrix entries of $rtA_{y}\coth(rtA_{y})$, each satisfying
an exponential bound
\begin{align}
|g_{k,d,I}(x) & |<c_{1}e^{c_{2}x}.\label{eq:function exponential bound}
\end{align}

\end{itemize}
Now since the kernels $H_{t}^{r}$ and $M_{t}^{r}$ both have the
same asymptotics as $t\rightarrow0$, Duhamel's principle, using \prettyref{eq:hdhm},
gives 

\begin{equation}
H_{t}^{r}-M_{t}^{r}=\int_{0}^{t}e^{-(t-s)D_{A_{r}}^{2}}\left\{ -(\partial_{s}+\mathcal{H})M_{s}^{r}-EM_{s}^{r}\right\} ds.\label{eq:ddhm}
\end{equation}
Now we substitute \prettyref{eq:hdsf} into \prettyref{eq:ddhm}.
Following this substitution, we use the heat kernel bound \prettyref{eq:mest},
\prettyref{eq:function exponential bound} and the bound 

\begin{equation}
|M_{t}^{r}(x,y)|\leq c_{3}e^{c_{4}rt}h_{t}(x,y)\label{eq:meest}
\end{equation}
 for constants $c_{3}$ and $c_{4}$ independent of $r$. These bounds
can be used to estimate the right hand side of \prettyref{eq:ddhm}
by a sum of finitely many terms of the form
\begin{equation}
c_{5}e^{c_{6}rt}\int_{0}^{t}ds\left(\int_{Y}h_{t-s}(z,x)s^{k}r^{d}\rho(x,y)^{I}h_{s}(x,y)dx\right),\label{eq:generic form integral}
\end{equation}
with each multi-index $(k,d,I)$ above satisfying \prettyref{eq:iin}.
Finally, \prettyref{eq:iin} and the inequalities

\begin{eqnarray}
\rho(x,y)^{I}h_{t}(x,y) & \leq & Ct^{\frac{1}{2}|I|}h_{2t}(x,y),\quad\label{eq:ehi1}\\
\int_{0}^{t}s^{-\frac{1}{2}}ds\left(\int_{Y}dxh_{2\left(t-s\right)}(z,x)h_{2s}(x,y)\right) & \leq & Ct^{\frac{1}{2}}h_{8t}(z,y),\label{eq:ehi2}
\end{eqnarray}
 (see \prettyref{sec:Estimates-on-Gaussian} for a proof of \prettyref{eq:ehi2})
give \prettyref{eq:mhest}.
\end{proof}

\subsection{Bound on the trace of $D_{A_{r}}e^{-tD_{A_{r}}^{2}}$}

We now turn to bound the pointwise $\textrm{tr}(D_{A_{r}}e^{-tD_{A_{r}}^{2}})$.
To this end, first consider the expansion for the heat kernel $H_{t}^{r}(x,y)$
given by

\begin{equation}
H_{t}^{r}(x,y)\sim\chi(\rho(x,y))h_{t}(x,y)\left(b_{0}^{r}(x,y)+b_{1}^{r}(x,y)t+b_{2}^{r}(x,y)t^{2}+\ldots\right).\label{eq:htre}
\end{equation}
Here the coefficients $b_{k}^{r}$ are smooth sections defined on
the neighborhood $\rho(x,y)<i_{g}$ of the diagonal in $Y\times Y$.
They are generated by solving a recursive system of transport equations
along geodesics as in chapter 7 of \cite{Roe}. The kernel $L_{t}^{r}(x,y)$
of $D_{A_{r}}e^{-tD_{A_{r}}^{2}}$ is simply $L_{t}^{r}=D_{A_{r}}H_{t}^{r}$.
It hence has an expansion given by

\begin{equation}
\begin{split}L_{t}^{r}(x,y)\sim & h_{t}(x,y)\left(c(d\chi)+c\left(-\frac{\rho d\rho}{2t}\right)\right)\left\{ b_{0}^{r}(x,y)+b_{1}^{r}(x,y)t+b_{2}^{r}(x,y)t^{2}+\ldots\right\} \\
 & +h_{t}(x,y)\left\{ D_{A_{r}}b_{0}^{r}(x,y)+D_{A_{r}}b_{1}^{r}(x,y)t+D_{A_{r}}b_{2}^{r}(x,y)t^{2}+\ldots\right\} .
\end{split}
\label{eq:dtre}
\end{equation}

By \prettyref{eq:cutoff}, $c(d\chi)$ is an endomorphism of the spin
bundle supported in the region where $\frac{i_{g}}{2}\leq\rho(x,y)\leq i_{g}.$
By \prettyref{eq:pointwise trace asymp as t->0} and \prettyref{eq:dtre},
the pointwise trace $\textrm{tr}(D_{A_{r}}e^{-tD_{A_{r}}^{2}})$ along
the diagonal has an expansion starting with a leading term of order
$t^{\frac{1}{2}}$. Since the restriction to the diagonal of the $c(d\chi)+c\left(-\frac{\rho d\rho}{2t}\right)$
term in \prettyref{eq:dtre} is zero, this implies that
\begin{align}
\textrm{tr}\left(D_{A_{r}}b_{k}^{r}(x,x)\right)=0, & \quad\textrm{for}\: k<\frac{n+1}{2},\label{eq:trace of coeff. L_t vanish}
\end{align}
at each point on the diagonal. To bound the trace of $L_{t}^{r}$
we will need a lemma giving a schematic form for the coefficients
$b_{k}^{r}(x,y)$. We again work in geodesic coordinates centered
at a point $y\in Y.$ Each heat kernel coefficient can be written
in terms of the frame \prettyref{eq:parallel transport frame} as
\begin{equation}
b_{k}^{r}=\sum_{\alpha\beta}f_{\alpha\beta,k}^{r}t_{\alpha\beta}\label{eq:expansion of hk coeff special frame}
\end{equation}
for some functions $f_{\alpha\beta,k}^{r}$. 
\begin{lem}
Each function $f_{\alpha\beta,k}^{r}$ appearing in \prettyref{eq:expansion of hk coeff special frame}
can be written as a finite sum

\begin{equation}
f_{\alpha\beta,k}^{r}=\sum_{(d,I)}r^{d}x^{I}f_{d,I}\label{eq:schf}
\end{equation}
 for some functions $f_{d,I}$, independent of $r$. Moreover, each
multi-index $(d,I)\in\mathbb{N}_{0}\times\mathbb{N}_{0}^{n}$ appearing
in \prettyref{eq:schf} satisfies the inequality 
\begin{equation}
d\leq k+\frac{1}{2}|I|.\label{eq:main index inequality}
\end{equation}
\end{lem}
\begin{proof}
The heat kernel coefficients $b_{k}^{r}(x,y)$ are given, as in chapter
7 of \cite{Roe}, by the recursion

\begin{eqnarray}
b_{0}^{r}(x,y) & = & \sum_{\alpha}g^{-1/4}(x)t_{\alpha\alpha},\label{eq:rec1}\\
b_{k}^{r}(x,y) & = & -\frac{1}{g^{1/4}(x)}\int_{0}^{1}\rho^{k-1}g^{1/4}(\rho x)D_{A_{r}}^{2}b_{k-1}(\rho x)d\rho,\quad\text{for}\quad k\geq1,\label{eq:rec2}
\end{eqnarray}
 where $g$ denotes the determinant of the metric on $Y$. Hence $b_{0}^{r}$
is clearly seen to be of the form \prettyref{eq:schf}. Equations
\prettyref{eq:dssf}-\prettyref{eq:ser} and \prettyref{eq:rec2}
imply that $b_{k}^{r}$ has the form \prettyref{eq:schf} assuming
it to be true for $b_{k-1}^{r}$. The lemma now follows by induction
on $k$.
\end{proof}
Following this we are ready to bound the pointwise trace $\textrm{tr}(D_{A_{r}}e^{-tD_{A_{r}}^{2}})$.
The above lemma will play an important role in the proposition below.
\begin{prop}
\label{prop:trace Dexp(td2) estimate}There exist constants $c_{1},c_{2}$,
independent of $r$, such that the pointwise trace $\textrm{tr}(D_{A_{r}}e^{-tD_{A_{r}}^{2}})$
satisfies the estimate

\begin{equation}
\norm{\textrm{tr}(D_{A_{r}}e^{-tD_{A_{r}}^{2}})}_{C^{0}}\leq c_{1}r^{\frac{n}{2}}e^{c_{2}rt},\label{eq:dtreb}
\end{equation}
 for all $t>0,\, r\geq1$.\end{prop}
\begin{proof}
Consider the remainder obtained after subtracting the first $\frac{n-1}{2}$
terms of the kernel expansion \prettyref{eq:dtre}

\begin{equation}
L_{t}^{r,\frac{n-1}{2}}=L_{t}^{r}-D_{A_{r}}(\chi h_{t}(b_{0}^{r}(x,y)+\ldots+t^{\frac{n-1}{2}}b_{\frac{n-1}{2}}^{r})).\label{eq:def remainder D(heat trace exp)}
\end{equation}
 From \prettyref{eq:dtre} and \prettyref{eq:trace of coeff. L_t vanish}
we see that 
\begin{align}
\textrm{tr}(L_{t}^{r})= & \textrm{tr}(L_{t}^{r,\frac{n-1}{2}}),\label{eq:tr(L_t)=00003Dtr(L_t^{(n-1)/2})}
\end{align}
and it hence suffices to bound $L_{t}^{r,\frac{n-1}{2}}$. We clearly
have $L_{t}^{r,\frac{n-1}{2}}=D_{A_{r}}H_{t}^{r,\frac{n-1}{2}}$ with
$H_{t}^{r,\frac{n-1}{2}}$ being the analogous remainder in the kernel
expansion for the heat trace

\begin{equation}
H_{t}^{r,\frac{n-1}{2}}=H_{t}^{r}-\chi h_{t}(b_{0}^{r}(x,y)+\ldots+t^{\frac{n-1}{2}}b_{\frac{n-1}{2}}^{r}).
\end{equation}
 Let us denote 
\begin{equation}
S_{t}^{r,\frac{n-1}{2}}=h_{t}(b_{0}^{r}(x,y)+\ldots+t^{\frac{n-1}{2}}b_{\frac{n-1}{2}}^{r}).\label{eq:formula for S_t}
\end{equation}
 The result of applying the heat operator to \prettyref{eq:def remainder D(heat trace exp)}
is then 
\begin{align}
(\partial_{t}+D_{A_{r}}^{2})(L_{t}^{r,\frac{n-1}{2}})= & Q_{t}^{r,\frac{n-1}{2}}+R_{t}^{r,\frac{n-1}{2}},\quad\textrm{where}\label{eq:heat operator applied to L}\\
Q_{t}^{r,\frac{n-1}{2}}= & -\chi(\partial_{t}+D_{A_{r}}^{2})D_{A_{r}}S_{t}^{r,\frac{n-1}{2}}\quad\textrm{and}\label{eq:formula for Q_t}\\
R_{t}^{r,\frac{n-1}{2}}= & -D_{A_{r}}^{2}\left\{ c(d\chi)S_{t}^{r,\frac{n-1}{2}}\right\} -D_{A_{r}}\left\{ c(d\chi)D_{A_{r}}S_{t}^{r,\frac{n-1}{2}}\right\} -\left\{ c(d\chi)D_{A_{r}}^{2}S_{t}^{r,\frac{n-1}{2}}\right\} .\label{eq:formula for R_t}
\end{align}
In other words, $R_{t}^{r,\frac{n-1}{2}}$ is the sum of the terms
obtained when some derivative differentiates the cutoff function $\chi$
in \prettyref{eq:def remainder D(heat trace exp)}. Now since $L_{t}^{r,\frac{n-1}{2}}\rightarrow0$
as $t\rightarrow0$, Duhamel's principle applied to \prettyref{eq:heat operator applied to L}
gives 
\begin{align}
L_{t}^{r,\frac{n-1}{2}}= & E_{t}^{r,\frac{n-1}{2}}+F_{t}^{r,\frac{n-1}{2}},\quad\textrm{where}\label{eq:L=00003DE+F}\\
E_{t}^{r,\frac{n-1}{2}}(z,y)= & \int_{0}^{t}ds\left(\int_{Y}dx\, H_{t-s}^{r}(z,x)Q_{s}^{r,\frac{n-1}{2}}(x,y)\right)\quad\textrm{and}\label{eq:formula E}\\
F_{t}^{r,\frac{n-1}{2}}(z,y)= & \int_{0}^{t}ds\left(\int_{Y}dx\, H_{t-s}^{r}(z,x)R_{s}^{r,\frac{n-1}{2}}(x,y)\right).\label{eq:formula F}
\end{align}
We first bound $F_{t}^{r,\frac{n-1}{2}}$, again working in geodesic
coordinates and the frame \prettyref{eq:parallel transport frame}.
Using \prettyref{eq:cdic} and the fact that Clifford multiplication
is parallel, the Dirac operator is seen to be of the form

\begin{equation}
D_{A}=A_{i}\partial_{i}+rx_{i}B_{i}+C,\label{eq:dic}
\end{equation}
where $A_{i}$, $B_{i}$ and $C$ are endomorphisms of $S\otimes L$
independent of $r$. Using \prettyref{eq:expansion of hk coeff special frame},\prettyref{eq:formula for S_t},\prettyref{eq:formula for R_t}
and \prettyref{eq:dic} we may write 
\begin{align}
R_{t}^{r,\frac{n-1}{2}}= & h_{t}\sum_{k=0}^{\frac{n-1}{2}}t^{k}\left(\sum f_{\alpha\beta,k}^{r,0}t_{\alpha\beta}\right)\label{eq:expansion R orth frame}
\end{align}
in the frame \prettyref{eq:parallel transport frame} for some coefficient
functions $f_{\alpha\beta}^{r,0}$. By \prettyref{eq:schf} each $f_{\alpha\beta}^{r,0}$
can be written as a finite sum
\begin{equation}
f_{\alpha\beta}^{r,0}=\sum_{(d,I)}r^{d}x^{I}f_{d,I}^{0}\label{eq:f0 sch form}
\end{equation}
for some functions $f_{d,I}^{0}$ independent of $r.$ Since $d\chi=0$
in a neighborhood of the diagonal, \prettyref{eq:formula for R_t}
implies each $f_{\alpha\beta}^{r,0}$ vanishes to infinite order near
the diagonal. Hence by \prettyref{eq:main index inequality} we may
assume that each multi-index $(d,I)$ in \prettyref{eq:f0 sch form}
satisfies $d\leq k+\frac{1}{2}|I|+1$. We now substitute \prettyref{eq:expansion R orth frame}
and \prettyref{eq:f0 sch form} in \prettyref{eq:formula F}. Using
this substitution along with the heat kernel bound \prettyref{eq:mest},
we may bound $|F_{t}^{r,\frac{n-1}{2}}(y,y)|$ by a sum of terms of
the form \prettyref{eq:generic form integral} each satisfying $d\leq k+\frac{1}{2}|I|+1$.
The inequalities \prettyref{eq:ehi1}-\prettyref{eq:ehi2} then give
the estimate
\begin{equation}
|F_{t}^{r,\frac{n-1}{2}}(y,y)|\leq c_{1}r^{\frac{n}{2}}e^{c_{2}rt}.\label{eq:F estimate}
\end{equation}

Next we estimate $E_{t}^{r,\frac{n-1}{2}}$. Following \prettyref{eq:formula for S_t}
and using the cancellations in the kernel expansion, resulting from
the transport equation, we see that
\begin{eqnarray}
Q_{t}^{r,\frac{n-1}{2}} & = & -\chi(\partial_{t}+D_{A_{r}}^{2})D_{A_{r}}S_{t}^{r,\frac{n-1}{2}}\\
 & = & \chi h_{t}t^{\frac{n-1}{2}}\left\{ -D_{A_{r}}^{3}b_{\frac{n-1}{2}}^{r}+c\left(\frac{\rho d\rho}{2t}\right)D_{A_{r}}^{2}b_{\frac{n-1}{2}}^{r}\right\} .
\end{eqnarray}
Equation \prettyref{eq:formula E} now gives 
\begin{multline}
E_{t}^{r,\frac{n-1}{2}}(y,y)=\int_{0}^{t}ds\Biggl(\int_{Y}dx\, H_{t-s}^{r}(y,x)\chi h_{s}(x,y)s^{\frac{n-1}{2}}\biggl\{-D_{A_{r}}^{3}b_{\frac{n-1}{2}}^{r}(x,y)\\
+c\left(\frac{\rho d\rho}{2s}\right)D_{A_{r}}^{2}b_{\frac{n-1}{2}}^{r}(x,y)\biggl\}\Biggl).\label{eq:rduha}
\end{multline}
We denote by $U_{t}^{r,\frac{n-1}{2}}$ and $V_{t}^{r,\frac{n-1}{2}}$
the kernels obtained by replacing $H_{t-s}^{r}$ in \prettyref{eq:rduha}
by $(H_{t-s}^{r}-M_{t-s}^{r})$ and $M_{t-s}^{r}$ respectively
\begin{multline}
U_{t}^{r,\frac{n-1}{2}}(y,y)=\int_{0}^{t}ds\Biggl(\int_{Y}dx\,\left(H_{t-s}^{r}(y,x)-M_{t-s}^{r}(y,x)\right)\chi h_{s}(x,y)s^{\frac{n-1}{2}}\biggl\{-D_{A_{r}}^{3}b_{\frac{n-1}{2}}^{r}(x,y)\\
+c\left(\frac{\rho d\rho}{2s}\right)D_{A_{r}}^{2}b_{\frac{n-1}{2}}^{r}(x,y)\biggl\}\Biggl),\label{eq:rduha U}
\end{multline}
\begin{multline}
V_{t}^{r,\frac{n-1}{2}}(y,y)=\int_{0}^{t}ds\Biggl(\int_{Y}dx\, M_{t-s}^{r}(y,x)\chi h_{s}(x,y)s^{\frac{n-1}{2}}\biggl\{-D_{A_{r}}^{3}b_{\frac{n-1}{2}}^{r}(x,y)\qquad\qquad\qquad\\
+c\left(\frac{\rho d\rho}{2s}\right)D_{A_{r}}^{2}b_{\frac{n-1}{2}}^{r}(x,y)\biggl\}\Biggl).\label{eq:rduha V}
\end{multline}
It is clear that 
\begin{align}
E_{t}^{r,\frac{n-1}{2}}= & U_{t}^{r,\frac{n-1}{2}}+V_{t}^{r,\frac{n-1}{2}}.\label{eq:E=00003DU+V}
\end{align}

We first bound $U_{t}^{r,\frac{n-1}{2}}$, again working in geodesic
coordinates and the frame \prettyref{eq:parallel transport frame}.
In terms of the orthonormal frame we may write

\begin{equation}
D_{A_{r}}^{3}b_{\frac{n-1}{2}}^{r}=\sum f_{\alpha\beta}^{r,1}t_{\alpha\beta},\quad D_{A_{r}}^{2}b_{\frac{n-1}{2}}^{r}=\sum f_{\alpha\beta}^{r,2}t_{\alpha\beta}\label{eq:dschf}
\end{equation}
 for some coefficient functions $f_{\alpha\beta}^{r,1}$ and $f_{\alpha\beta}^{r,2}$.
Using \prettyref{eq:schf} and \prettyref{eq:dic} these can be expressed
as finite sums 
\begin{align}
f_{\alpha\beta}^{r,1}=\sum_{(d,I)\in S_{1}}x^{I}r^{d}f_{d,I}^{1},\label{eq:schematic form coeffs 1}\\
f_{\alpha\beta}^{r,2}=\sum_{(d,I)\in S_{2}}x^{I}r^{d}f_{d,I}^{2},\label{eq:schematic form coeffs 2}
\end{align}
where $f_{d,I}^{1},f_{d,I}^{2}$ are functions independent of $r$
and $S_{1},S_{2}$ are finite subsets of $\mathbb{N}_{0}\times\mathbb{N}_{0}^{n}$.
Moreover, \prettyref{eq:main index inequality} now gives
\begin{align}
d\leq\frac{n-1}{2}+\frac{1}{2}|I|+\frac{3}{2} & \qquad\forall(d,I)\in S_{1}\quad\text{and}\label{eq:inex1}\\
d\leq\frac{n-1}{2}+\frac{1}{2}|I|+1 & \qquad\forall(d,I)\in S_{2}.\label{eq:inex2}
\end{align}
Again we substitute \prettyref{eq:dschf}-\prettyref{eq:schematic form coeffs 2}
into \prettyref{eq:rduha U}. This substitution, the bound \prettyref{eq:mhest},
along with the inequalities \prettyref{eq:ehi1}-\prettyref{eq:ehi2}
and \prettyref{eq:schematic form coeffs 1}-\prettyref{eq:schematic form coeffs 2}
now give the estimate

\begin{equation}
|U_{t}^{r,\frac{n-1}{2}}(y,y)|\leq c_{3}r^{\frac{n}{2}}e^{c_{4}rt}.\label{eq:uest}
\end{equation}

Next we estimate $V_{t}^{r,\frac{n-1}{2}}$. First we use a Taylor
expansion to write

\begin{equation}
f_{d,I}^{1}=g_{d,I}^{1}+x_{i}h_{d,I,i}^{1}\quad\text{and}\quad f_{d,I}^{2}=g_{d,I}^{2}+x_{i}h_{d,I,i}^{2}\label{eq:Taylor expansion odd even}
\end{equation}
 where each of $g_{d,I}^{1}$ and $g_{d,I}^{2}$ is an even function
in these coordinates. We now let 
\begin{align}
\bar{S_{1}}= & \left\{ (d,I)\in S_{1}|d=\frac{n-1}{2}+\frac{1}{2}|I|+\frac{3}{2}\right\} ,\label{eq:S1 bar}\\
\bar{S_{2}}= & \left\{ (d,I)\in S_{2}|d=\frac{n-1}{2}+\frac{1}{2}|I|+1\right\} ,\label{eq:S2 bar}
\end{align}
and define 
\begin{align}
\bar{f}_{\alpha\beta}^{r,1}=\sum_{(d,I)\in\bar{S}_{1}}x^{I}r^{d}g_{d,I}^{1},\quad & \tilde{f}_{\alpha\beta}^{r,1}=\sum_{(d,I)\in S_{1}\setminus\bar{S}_{1}}x^{I}r^{d}g_{d,I}^{1}+\sum_{(d,I)\in S_{1}}x^{I}r^{d}\left(x_{i}h_{d,I,i}^{1}\right),\label{eq:splitting f_ab 1}\\
\bar{f}_{\alpha\beta}^{r,2}=\sum_{(d,I)\in\bar{S}_{2}}x^{I}r^{d}g_{d,I}^{2},\quad & \tilde{f}_{\alpha\beta}^{r,2}=\sum_{(d,I)\in S_{2}\setminus\bar{S}_{2}}x^{I}r^{d}g_{d,I}^{2}+\sum_{(d,I)\in S_{2}}x^{I}r^{d}\left(x_{i}h_{d,I,i}^{2}\right).\label{eq:splitting f_ab 2}
\end{align}
Clearly, by \prettyref{eq:schematic form coeffs 1}-\prettyref{eq:schematic form coeffs 2}
and \prettyref{eq:Taylor expansion odd even}-\prettyref{eq:splitting f_ab 2},
\begin{align*}
f_{\alpha\beta}^{r,1}=\bar{f}_{\alpha\beta}^{r,1}+\tilde{f}_{\alpha\beta}^{r,1}\quad\textrm{and}\quad & f_{\alpha\beta}^{r,2}=\bar{f}_{\alpha\beta}^{r,2}+\tilde{f}_{\alpha\beta}^{r,2}.
\end{align*}

Next we claim that the contribution of $\bar{f}_{\alpha\beta}^{r,1}$
to $V_{t}^{r,\frac{n-1}{2}}$ is zero. To see this, first observe
that $(d,I)\in\bar{S}_{1}$ implies $|I|$ is odd by \prettyref{eq:S1 bar}.
Hence $\bar{f}_{\alpha\beta}^{r,1}$ is an odd function, using \prettyref{eq:splitting f_ab 1}
and the fact that $g_{d,I}^{1}$ is even. Hence the integral corresponding
to $\bar{f}_{\alpha\beta}^{r,1}$ in \prettyref{eq:rduha V} is zero,
being the integral of an odd function in these coordinates. Similarly,
we claim that the contribution of $\bar{f}_{\alpha\beta}^{r,2}$ to
$V_{t}^{r,\frac{n-1}{2}}$ is zero. This time, $(d,I)\in\bar{S}_{2}$
implies $|I|$ is even by \prettyref{eq:S2 bar}. Hence $\bar{f}_{\alpha\beta}^{2}$
is an even function using \prettyref{eq:splitting f_ab 2}. However
the integral corresponding to $\bar{f}_{\alpha\beta}^{r,2}$ in \prettyref{eq:rduha V}
is still the integral of an odd function in these coordinates, because
of the $c\left(\frac{\rho d\rho}{2s}\right)$ term in \prettyref{eq:rduha V}. 

Following this the contribution of $\tilde{f}_{\alpha\beta}^{r,1}$
to $V_{t}^{r,\frac{n-1}{2}}(y,y)$ can be bounded by a finite sum
of terms of the form 
\begin{align*}
c_{1}e^{c_{2}rt}\int_{0}^{t}ds\left(\int_{Y}h_{2\left(t-s\right)}(y,x)s^{\frac{n-1}{2}}r^{d}\rho(x,y)^{I}h_{s}(x,y)dx\right), & \quad\textrm{with}\quad d\leq\frac{n-1}{2}+\frac{1}{2}|I|+1.
\end{align*}
Again using the inequalities \prettyref{eq:ehi1}-\prettyref{eq:ehi2},
and estimating the contribution of $\tilde{f}_{\alpha\beta}^{r,2}$
in similar fashion, gives the estimate

\begin{equation}
|V_{t}^{\frac{n-1}{2}}(y,y)|\leq c_{5}r^{\frac{n}{2}}e^{c_{6}rt}.\label{eq:vest}
\end{equation}
 Following \prettyref{eq:uest}, \prettyref{eq:vest} and \prettyref{eq:E=00003DU+V}
we obtain the estimate

\begin{equation}
|E_{t}^{\frac{n-1}{2}}(y,y)|\leq c_{7}r^{\frac{n}{2}}e^{c_{8}rt},\label{eq:E est}
\end{equation}
 for constants $c_{5}$ and $c_{6}$ independent of $r$. Equations
\prettyref{eq:L=00003DE+F}, \prettyref{eq:F estimate} and \prettyref{eq:E est}
then give 
\begin{equation}
|L_{t}^{\frac{n-1}{2}}(y,y)|\leq c_{9}r^{\frac{n}{2}}e^{c_{10}rt},\label{eq:L est}
\end{equation}
for constants $c_{9}$ and $c_{10}$ independent of $r$. The proposition
now follows from \prettyref{eq:tr(L_t)=00003Dtr(L_t^{(n-1)/2})} and
\prettyref{eq:L est}.
\end{proof}

\section{Asymptotics of the spectral measure\label{sec:Asymptotics-}}

We now consider the rescaled operator $D_{r}=\frac{1}{\sqrt{r}}D_{A_{r}}$.
Here we shall find the asymptotics of its spectral measure using the
heat trace estimates of the previous section. We first consider the
heat traces of $D_{r}$. 
\begin{thm}
For any $t>0$,

\begin{eqnarray}
\,\lim_{r\rightarrow\infty}r^{-\frac{n}{2}}\textrm{tr}(e^{-tD_{r}^{2}}) & =\frac{1}{\left(4\pi t\right)^{\frac{n}{2}}}\int_{Y}\det{}^{\frac{1}{2}}\left(\frac{tA_{y}}{\tanh tA_{y}}\right)dy & \quad\textrm{and}\label{eq:rescaled heat trace limit}\\
\lim_{r\rightarrow\infty}r^{-n/2}\textrm{tr}\left(D_{r}e^{-tD_{r}^{2}}\right) & = & 0,\label{eq:rescaled tr(Dexp(tD2)) estimate}
\end{eqnarray}
where $A$ is as defined by \prettyref{eq:antisymm endomorphism A}.
The convergences above are uniform in compact intervals of $t\in\mathbb{R}_{>0},\mathbb{R}_{\geq0}$
respectively.\end{thm}
\begin{proof}
If $H'_{t}$ denotes the kernel of $e^{-tD_{r}^{2}}$ it is clear
that $H'_{t}=H_{\frac{t}{r}}$ after rescaling. Hence proposition
\prettyref{prop:Heat Mehler comparison} gives the estimate
\begin{equation}
|H'_{t}-M_{\frac{t}{r}}|\leq c_{1}h_{\frac{8t}{r}}t^{\frac{1}{2}}r^{-\frac{1}{2}}e^{c_{2}t}\label{eq:rescaled heat kernel bound}
\end{equation}
for some constants $c_{1,}c_{2}$ independent of $r$. Hence $\textrm{tr}(e^{-tD_{r}^{2}})-\textrm{tr}(M_{\frac{t}{r}})=O(r^{\frac{n-1}{2}}),\,\forall t>0$.
It now remains to compute $\textrm{tr}(M_{\frac{t}{r}})$ in order
to prove \prettyref{eq:rescaled heat trace limit}. By \prettyref{eq:Mehler's kernel def},
the highest order part in $r$ of $\textrm{tr}(M_{\frac{t}{r}})$
is given by
\[
\frac{1}{(4\pi t)^{\frac{n}{2}}}\int_{Y}\det{}^{\frac{1}{2}}\left(\frac{tA_{y}}{\sinh tA_{y}}\right)\textrm{tr}\left(e^{-tc(da)}\right)dy.
\]
Since $A$ is an antisymmetric endomorphism it maybe diagonalized
to give an orthonormal basis $e,e_{1},f_{1},\ldots,e_{m},f_{m}\;$
of $T_{y}Y$ with eigenvalues $0,\pm\lambda_{1},\ldots,\pm\lambda_{m}\;(\lambda_{i}\geq0)$
respectively. We hence have $da=i\left(\lambda_{1}e_{1}\wedge f_{1}+\ldots+\lambda_{m}e_{m}\wedge f_{m}\right)$
and 
\begin{equation}
e^{-tc(da)}=\prod_{k=1}^{m}\left(\cosh(t\lambda_{k})-ie_{k}f_{k}\sinh(t\lambda_{k})\right).\label{eq:Exponential Clifford multiplication}
\end{equation}
Now if $I\in\left\{ 2,\ldots,m\right\} $, we have the commutation
\[
e_{1}f_{1}\left(\prod_{i\in I}e_{k}f_{k}\right)=\frac{1}{2}\left[e_{1},f_{1}\left(\prod_{i\in I}e_{k}f_{k}\right)\right].
\]
This shows that the only traceless terms in the expansion of \prettyref{eq:Exponential Clifford multiplication}
are the constants and hence $\textrm{tr}\left(e^{-tc(da)}\right)=\cosh(tA_{y})$.
Equation \prettyref{eq:rescaled heat trace limit} now follows. For
the second part of the theorem note that proposition \prettyref{prop:trace Dexp(td2) estimate}
gives the estimate
\begin{align}
|\textrm{tr}(D_{r}e^{-tD_{r}^{2}})|\leq & c_{1}r^{\frac{n-1}{2}}e^{c_{2}t}\label{eq:rescaled trace bound}
\end{align}
for uniform constants $c_{1},c_{2}$ independent of $r$. From this
equation, \prettyref{eq:rescaled tr(Dexp(tD2)) estimate} follows.
The uniformity of the limits \prettyref{eq:rescaled heat trace limit},\prettyref{eq:rescaled tr(Dexp(tD2)) estimate}
is also an easy consequence of the estimates \prettyref{eq:rescaled heat kernel bound}
and \prettyref{eq:rescaled trace bound}.
\end{proof}
The above theorem also follows from the rescaling argument as described
in section c) of \cite{Bismut-Vasserot}.

Next we let $N_{r}(\sigma$) denote the number of eigenvalues of $D_{r}$
in the interval $[-\sigma,\sigma]$. We also use the notation $\left\langle x\right\rangle =\sqrt{1+|x|^{2}}$
for any $x\in\mathbb{R}^{d}$. We shall need the following estimates. 
\begin{prop}
Let $\varphi\in\mathcal{S}$ be a Schwartz function. Then we have
the estimates

\begin{eqnarray}
N_{r}(\sigma) & \leq & c_{1}r^{\frac{n}{2}}(1+\sigma^{2})^{\frac{n}{2}}\label{eq:Counting function bound}\\
\textrm{tr}\,\varphi(D_{r}) & \leq & c_{2}r^{\frac{n}{2}}\left\Vert \left\langle x\right\rangle ^{n+2}\varphi\right\Vert _{C^{0}}\label{eq:Functional trace bound}
\end{eqnarray}

for constants $c_{1},c_{2}$ independent of $r$. In case the function
$\varphi$ is odd, we have 
\begin{equation}
\lim_{r\rightarrow\infty}r^{-\frac{n}{2}}\textrm{tr}\,\varphi(D_{r})=0.\label{eq:odd trace cancellation}
\end{equation}
 \end{prop}
\begin{proof}
We begin with 
\[
N_{r}(\sigma)e^{-t\sigma^{2}}\leq\textrm{tr}(e^{-tD_{r}^{2}})=\textrm{tr}(H_{\frac{t}{r}})\leq c_{1}\left(\frac{r}{t}\right)^{\frac{n}{2}}e^{c_{2}t}
\]
using proposition \prettyref{prop:Maximum Principle estimate}. This
gives $N_{r}(\sigma)\leq c_{1}\left(\frac{r}{t}\right)^{\frac{n}{2}}e^{(c_{2}+1)(\sigma^{2}+1)t}$
from which\prettyref{eq:Counting function bound} follows on substituting
$t=\frac{1}{\sigma^{2}+1}$. 

For the second part we estimate
\begin{eqnarray*}
\left|\textrm{tr}\,\varphi(D_{r})\right| & = & \left|\sum_{k=-\infty}^{\infty}\sum_{\lambda\in Spec(D_{r})\cap[k,k+1)}\varphi(\lambda)\right|\\
 & \leq c_{1} & \sum_{k=-\infty}^{\infty}\sum_{\lambda\in Spec(D_{r})\cap[k,k+1)}\frac{\left\langle \lambda\right\rangle ^{n+2}}{\left\langle k+1\right\rangle ^{n+2}}\left|\varphi(\lambda)\right|\\
 & \leq & c_{1}\left\Vert \left\langle x\right\rangle ^{n+2}\varphi\right\Vert _{C^{0}}\left(\sum_{k=-\infty}^{\infty}\frac{N_{r}(k+1)}{\left\langle k+1\right\rangle ^{n+2}}\right)\\
 & \leq c_{2} & r^{\frac{n}{2}}\left\Vert \left\langle x\right\rangle ^{\frac{n}{2}+2}\varphi\right\Vert _{C^{0}}\left(\sum_{k=0}^{\infty}\frac{1}{\left\langle k+1\right\rangle ^{2}}\right).
\end{eqnarray*}
Finally to prove the third part note that \prettyref{eq:odd trace cancellation}
is true for the family of odd Schwartz functions $\varphi_{t}=xe^{-tx^{2}}$
on account of \prettyref{eq:rescaled tr(Dexp(tD2)) estimate}. Since
the convergence in \prettyref{eq:rescaled tr(Dexp(tD2)) estimate}
is uniform it may be differentiated to obtain \prettyref{eq:odd trace cancellation}
for the odd functions $\varphi=x^{2m+1}e^{-x^{2}},m\in\mathbb{N}_{0}$.
Now \prettyref{eq:Functional trace bound} along with the fact that
the span of $\left\{ x^{2m+1}e^{-x^{2}}\right\} $ is dense in the
space of odd Schwartz functions gives \prettyref{eq:odd trace cancellation}.
\end{proof}
Now consider the rescaled spectral measure of $D_{r}^{2}$ given by
\[
\mu_{r}=r^{-\frac{n}{2}}\sum_{\lambda\in Spec(D_{r})}\delta_{\lambda^{2}}.
\]
Let $C_{0}^{0}\left(\mathbb{R}_{\geq0}\right)$ denote the space of
bounded, continuous functions on $\mathbb{R}_{\geq0}$ vanishing at
$\infty$. Consider the Banach space $\mathcal{B}=\left\langle x\right\rangle ^{-n-2}$$C_{0}^{0}\left(\mathbb{R}_{\geq0}\right)$.
By a measure here we shall mean an element of the dual space $\mathcal{B}'$.
By \prettyref{eq:Functional trace bound} $\mu_{r}$ is a family of
uniformly bounded measures in $\mathcal{B}'$. We now derive a formula
for the limit of the measures $\mu_{r}$ as $r\rightarrow\infty$.
Expecting the Laplace transform of the limit measure to be given by
the integral in \prettyref{eq:rescaled heat trace limit}, we find
the Laplace inverse of its integrand. Let $2n_{y}+1$ be the dimension
of the kernel of $A_{y}$ at any point $y\in Y$ and $m_{y}=\frac{n-(2n_{y}+1)}{2}$.
Let $u(s)$ be the Heaviside step function and $Z(k)$ denote the
number of non-zero components of a multi-index $k\in\mathbb{N}_{0}^{m_{y}}$.
We then compute
\begin{eqnarray*}
\frac{1}{(4\pi t)^{\frac{n}{2}}}\det{}^{\frac{1}{2}}\left(\frac{tA}{\tanh tA}\right) & = & \frac{t^{-n_{y}-\frac{1}{2}}}{(4\pi)^{\frac{n}{2}}}\prod_{\lambda_{i}^{y}>0}\frac{\lambda_{i}^{y}}{\tanh(t\lambda_{i}^{y})}\\
 & = & \frac{t^{-n_{y}-\frac{1}{2}}}{(4\pi)^{\frac{n}{2}}}\prod_{\lambda_{i}^{y}>0}\lambda_{i}^{y}(1+2e^{-2t\lambda_{i}}+2e^{-4t\lambda_{i}}+\ldots)\\
 & = & \frac{t^{-n_{y}-\frac{1}{2}}}{(4\pi)^{\frac{n}{2}}}\left(\prod_{\lambda_{i}^{y}>0}\lambda_{i}^{y}\right)\left(\sum_{k\in\mathbb{N}_{0}^{m_{y}}}2^{Z(k)}e^{-2tk\cdot\lambda}\right)\\
 & = & \frac{\left(\prod_{\lambda_{i}^{y}>0}\lambda_{i}^{y}\right)}{(4\pi)^{\frac{n}{2}}}\mathcal{L}_{s\rightarrow t}\left(s^{n_{y}-\frac{1}{2}}\ast\left(\sum_{k\in\mathbb{N}_{0}^{m_{y}}}2^{Z(k)}\delta(s-2k\cdot\lambda)\right)\right)\\
 & = & \frac{\left(\prod_{\lambda_{i}^{y}>0}\lambda_{i}^{y}\right)}{(4\pi)^{\frac{n}{2}}}\mathcal{L}_{s\rightarrow t}\left(\sum_{k\in\mathbb{N}_{0}^{m_{y}}}2^{Z(k)}u(s-2k\cdot\lambda)(s-2k\cdot\lambda)^{n_{y}-\frac{1}{2}}\right).
\end{eqnarray*}

Motivated by this pointwise calculation on $Y$ we define the measure
\[
\mu_{\infty}^{y}=\frac{\left(\prod_{\lambda_{i}^{y}>0}\lambda_{i}^{y}\right)}{(4\pi)^{\frac{n}{2}}}\left(\sum_{k\in\mathbb{N}_{0}^{m_{y}}}2^{Z(k)}u(s-2k\cdot\lambda)(s-2k\cdot\lambda)^{n_{y}-\frac{1}{2}}\right).
\]
We now have the following proposition.
\begin{prop}
\label{prop:Pointwise limit measures family continuity} Each $\mu_{\infty}^{y}$
is a measure in $\mathcal{B}'$ satisfying $\left\Vert \mu_{\infty}^{y}\right\Vert _{\mathcal{B}'}\leq C$
for some uniform constant $C$ independent of $y$. Furthermore, the
family of measures $\mu_{\infty}^{y}\in\mathcal{B}'$ is weakly continuous
in $y$.\end{prop}
\begin{proof}
For $\varphi\in\mathcal{B}$ we estimate

\begin{eqnarray}
\left|\mu_{\infty}^{y}(\varphi)\right| & = & \left|\frac{\left(\prod_{\lambda_{i}^{y}>0}\lambda_{i}^{y}\right)}{(4\pi)^{\frac{n}{2}}}\left(\sum_{k\in\mathbb{N}_{0}^{m_{y}}}2^{Z(k)}\left(\int_{2k\cdot\lambda}^{\infty}\varphi(s)(s-2k\cdot\lambda)^{n_{y}-\frac{1}{2}}ds\right)\right)\right|\nonumber \\
 & \leq & \frac{\left(\prod_{\lambda_{i}^{y}>0}\lambda_{i}^{y}\right)}{(2\pi)^{\frac{n}{2}}}\left\Vert \varphi\right\Vert _{\mathcal{B}}\left(\sum_{k\in\mathbb{N}_{0}^{m_{y}}}\left(\int_{2k\cdot\lambda}^{\infty}\left\langle s\right\rangle ^{-\frac{n}{2}-2}(s-2k\cdot\lambda)^{n_{y}-\frac{1}{2}}ds\right)\right)\nonumber \\
 & = & \frac{\left(\prod_{\lambda_{i}^{y}>0}\lambda_{i}^{y}\right)}{(2\pi)^{\frac{n}{2}}}\left\Vert \varphi\right\Vert _{\mathcal{B}}\left(\sum_{k\in\mathbb{N}_{0}^{m_{y}}}\left(\int_{0}^{\infty}\frac{s{}^{n_{y}-\frac{1}{2}}}{\left\langle s+2k\cdot\lambda\right\rangle ^{n_{y}+m_{y}+2+\frac{1}{2}}}ds\right)\right)\nonumber \\
 & \leq & \frac{\left(\prod_{\lambda_{i}^{y}>0}\lambda_{i}^{y}\right)}{(2\pi)^{\frac{n}{2}}}\left\Vert \varphi\right\Vert _{\mathcal{B}}\left(\sum_{k\in\mathbb{N}_{0}^{m_{y}}}\frac{1}{\left\langle 2k\cdot\lambda\right\rangle ^{m_{y}+\frac{3}{2}}}\left(\int_{0}^{\infty}\frac{s{}^{-\frac{1}{2}}}{\left\langle s\right\rangle }ds\right)\right)\nonumber \\
 & \leq & C\left\Vert \varphi\right\Vert _{\mathcal{B}}\left(\prod_{\lambda_{i}^{y}>0}\lambda_{i}^{y}\right)\left(\sum_{k\in\mathbb{N}_{0}^{m_{y}}}\frac{1}{\left\langle 2k\cdot\lambda\right\rangle ^{m_{y}+\frac{3}{2}}}\right).\label{eq:uniform bound pointwise measures}
\end{eqnarray}
Now if $\sup_{y\in Y}\left\Vert A_{y}\right\Vert =\alpha$ then each
$\lambda_{i}^{y}\leq\alpha$. If $N_{R}$ denotes the cardinality
of the set $S_{R}=\left\{ k\in\mathbb{N}_{0}^{m_{y}}|2k\cdot\lambda\leq R\right\} $
we have the bound $N_{R}\leq C(R+\alpha)^{m_{y}}\left(\prod_{\lambda_{i}^{y}>0}\lambda_{i}^{y}\right)^{-1}$
for $C$ depending only $n$. This is obtained on observing that the
union of the $m_{y}$-parallelotope's based at points of $S_{R}$
can be covered by the appropriately large ball. Hence we may estimate
\prettyref{eq:uniform bound pointwise measures} further by
\begin{eqnarray}
 &  & C\left\Vert \varphi\right\Vert _{\mathcal{B}}\left(\prod_{\lambda_{i}^{y}>0}\lambda_{i}^{y}\right)\left(\sum_{k\in\mathbb{N}_{0}^{m_{y}}}\frac{1}{\left\langle 2k\cdot\lambda\right\rangle ^{m_{y}+\frac{3}{2}}}\right)\nonumber \\
 & \leq & C\left\Vert \varphi\right\Vert _{\mathcal{B}}\left(\prod_{\lambda_{i}^{y}>0}\lambda_{i}^{y}\right)\left(\sum_{l\in\mathbb{N}_{0}}\frac{N_{l+1}}{\left\langle l\right\rangle ^{m_{y}+\frac{3}{2}}}\right)\nonumber \\
 & \leq & C_{1}\left\Vert \varphi\right\Vert _{\mathcal{B}}\left(\sum_{l\in\mathbb{N}_{0}}\frac{(l+\alpha)^{m_{y}}}{\left\langle l\right\rangle ^{m_{y}+\frac{3}{2}}}\right)\nonumber \\
 & \leqq & C_{2}\left\Vert \varphi\right\Vert _{\mathcal{B}}.\label{eq:uniform bound pointwise measures end}
\end{eqnarray}
Now we prove the second part of the proposition. By \prettyref{eq:uniform bound pointwise measures end}
each $\mu_{\infty}^{y}$ lies in $\mathcal{B}'$ and we need to show
that $\mu_{\infty}^{y}(\varphi)$ is a continuous function of $y$
for every $\varphi\in\mathcal{B}$. First note that 
\begin{equation}
\mu_{\infty}^{y}(e^{-ts})=\frac{1}{(4\pi t)^{\frac{n}{2}}}\det{}^{\frac{1}{2}}\left(\frac{tA_{y}}{\sinh tA_{y}}\right)\label{eq:pointwise limit measure definition}
\end{equation}
by construction. Hence $\mu_{\infty}^{y}(e^{-s})$ is a continuous
function of $y$. By differentiating \prettyref{eq:pointwise limit measure definition}
further we obtain that $\mu_{\infty}^{y}(s^{m}e^{-s})$ is a continuous
function of $y$ for each $m\in\mathbb{N}_{0}$. The result now follows
on knowing that $\left\Vert \mu_{\infty}^{y}\right\Vert _{\mathcal{B}'}$
is uniformly bounded in $y$ and the span of $s^{m}e^{-s}$ is dense
in $\mathcal{B}$.
\end{proof}
Next we define the measure $\mu_{\infty}$ via 
\[
\mu_{\infty}(\varphi)=\int_{Y}\mu_{\infty}^{y}(\varphi)dy.
\]
By proposition \prettyref{prop:Pointwise limit measures family continuity},
we have that $\mu_{\infty}$ is a well defined element of $\mathcal{B}'$. 
\begin{prop}
We have the weak convergence $\mu_{r}\rightharpoonup\mu_{\infty}$
in $\mathcal{B}'$.\end{prop}
\begin{proof}
By \prettyref{eq:rescaled heat trace limit} we have that $\mu_{r}(e^{-ts})\rightarrow\mu_{\infty}(e^{-ts})$
and since this limit is uniform on compact intervals of time it may
be differentiated to obtain $\mu_{r}(s^{m}e^{-s})\rightarrow\mu_{\infty}(s^{m}e^{-s}),\forall m\in\mathbb{N}_{0}$.
Weak convergence again follows on knowing that $\left\Vert \mu_{r}\right\Vert _{\mathcal{B}'}$
is uniformly bounded in $r$ and the span of $s^{m}e^{-s}$ is dense
in $\mathcal{B}$. 
\end{proof}
Finally, we shall need the following information about the limit measure.
\begin{prop}
Let $\varphi\in\mathcal{B},\,0\leq\varphi\leq1$ be such that $supp(\varphi)\subseteq[0,\epsilon],\,0<\epsilon<1$.
Then we have
\begin{equation}
\mu_{\infty}(\varphi)\leq C\epsilon^{\frac{1}{2}}\label{eq:estimate spectral measure near zero}
\end{equation}
for some constant $C$ independent of $\varphi,\epsilon$. \end{prop}
\begin{proof}
We estimate
\begin{eqnarray*}
\left|\mu_{\infty}^{y}(\varphi)\right| & = & \left|\frac{\left(\prod_{\lambda_{i}^{y}>0}\lambda_{i}^{y}\right)}{(4\pi)^{\frac{n}{2}}}\left(\sum_{k\in\mathbb{N}_{0}^{m_{y}}}2^{Z(k)}\left(\int_{2k\cdot\lambda}^{\infty}\varphi(s)(s-2k\cdot\lambda)^{n_{y}-\frac{1}{2}}ds\right)\right)\right|\\
 & \leq & c_{1}\left(\prod_{\lambda_{i}^{y}>0}\lambda_{i}^{y}\right)\left(\sum_{k\in\mathbb{N}_{0}^{m_{y}}}\left(\int_{0}^{\infty}\varphi(u+2k\cdot\lambda)u{}^{n_{y}-\frac{1}{2}}du\right)\right)\\
 & \leq & c_{2}\left(\prod_{\lambda_{i}^{y}>0}\lambda_{i}^{y}\right)\left(\sum_{2k\cdot\lambda\leq\epsilon}\left(\int_{0}^{\epsilon}u{}^{n_{y}-\frac{1}{2}}du\right)\right)\\
 & \leq & c_{2}\left(\prod_{\lambda_{i}^{y}>0}\lambda_{i}^{y}\right)\frac{\epsilon^{n_{y}+\frac{1}{2}}}{n_{y}+\frac{1}{2}}N_{\epsilon}\\
 & \leq & c_{3}\epsilon^{n_{y}+\frac{1}{2}}(\epsilon+\alpha)^{m_{y}}\\
 & \leq & c_{4}\epsilon^{\frac{1}{2}}.
\end{eqnarray*}
For a constant $c_{4}$ independent of $y$. The proposition now follows
on integration over $Y$.
\end{proof}
We are now ready to give the proof of \prettyref{thm:main theorem}
below. 
\begin{proof}[Proof of \prettyref{thm:main theorem}]

Since the eta invariant is unchanged under rescaling it suffices to
consider $\bar{\eta}(D_{r})$. We then write
\begin{eqnarray}
\bar{\eta}(D_{r}) & = & \frac{1}{2}\left\{ \textrm{dim\,\ ker}\,(D_{r})+\intop_{0}^{\infty}\frac{1}{\sqrt{\pi t}}\textrm{tr}\left(D_{r}e^{-tD_{r}^{2}}\right)dt\right\} \nonumber \\
 & = & \frac{1}{2}\left\{ \intop_{0}^{1}\frac{1}{\sqrt{\pi t}}\textrm{tr}\left(D_{r}e^{-tD_{r}^{2}}\right)dt+\textrm{tr}\, E(D_{r})\right\} .\label{eq:eta invariant splitting}
\end{eqnarray}
Here $E(x)=\text{sign}(x)\text{erfc}(|x|)=\text{sign}(x)\cdot\frac{2}{\sqrt{\pi}}\int_{|x|}^{\infty}e^{-s^{2}}ds$
with the convention $\text{sign}(0)=1$. The first summand of \prettyref{eq:eta invariant splitting}
is $o(r^{\frac{n}{2}})$ on account of the uniform convergence in
\prettyref{eq:rescaled tr(Dexp(tD2)) estimate}. To bound the second
term we fix $0<\epsilon<1$ and define Schwartz functions $\varphi_{1},\varphi_{2}\in\mathcal{S}$
satisfying.
\begin{enumerate}
\item $\varphi_{1}$ odd, $\varphi_{2}$ even
\item $-1\leq\varphi_{1}\leq1,\;0\leq\varphi_{2}\leq1$
\item $\varphi_{1}(x)=E(x)$ for $x\notin\left[-\frac{\epsilon}{2},\frac{\epsilon}{2}\right]$
\label{-for-odd  function approximating E}
\item $supp(\varphi_{2})\subset\left[-\epsilon,\epsilon\right]$,$\varphi_{2}(x)=1$
for $x\in\left[-\frac{\epsilon}{2},\frac{\epsilon}{2}\right].$\label{,-for cutoff near zero }
\end{enumerate}

Since $\varphi_{2}$ is even we may also assume $\varphi_{2}(x)=\bar{\varphi}_{2}(x^{2})$.
We then have
\begin{eqnarray*}
r^{-\frac{n}{2}}\textrm{tr}\, E(D_{r}) & \leq & r^{-\frac{n}{2}}\left(\left|\textrm{tr}\, E(D_{r})-\textrm{tr}\,\varphi_{1}(D_{r})\right|+\left|\textrm{tr}\,\varphi_{1}(D_{r})\right|\right)\\
 & \leq & r^{-\frac{n}{2}}\left(2N_{r}\left(\frac{\epsilon}{2}\right)+\left|\textrm{tr}\,\varphi_{1}(D_{r})\right|\right)\qquad(\text{by \ref{-for-odd  function approximating E}})\\
 & \leq & r^{-\frac{n}{2}}\left(2\textrm{tr}\,\varphi_{2}(D_{r})+\left|\textrm{tr}\,\varphi_{1}(D_{r})\right|\right)\qquad(\text{by }\ref{,-for cutoff near zero })\\
 & = & 2\mu_{r}(\bar{\varphi}_{2})+r^{-\frac{n}{2}}\left|\textrm{\textrm{tr}}\,\varphi_{1}(D_{r})\right|\\
 & \leq & 2\left|\mu_{r}(\bar{\varphi}_{2})-\mu_{\infty}(\bar{\varphi}_{2})\right|+2\left|\mu_{\infty}(\bar{\varphi}_{2})\right|+r^{-\frac{n}{2}}\left|\textrm{\textrm{tr}}\,\varphi_{1}(D_{r})\right|\\
 & \leq & c\epsilon^{\frac{1}{4}},
\end{eqnarray*}
for $r$ sufficiently large by \prettyref{eq:odd trace cancellation},
\prettyref{eq:estimate spectral measure near zero} and the weak convergence
$\mu_{r}\rightharpoonup\mu_{\infty}$. 

\end{proof}

\section{Eta invariant of a circle bundle\label{sec:Eta-invariant-of}}

In this section we consider the eta invariant in a specialized case.
In particular, we let $Y$ to be the total space of principal circle
bundle $S^{1}\rightarrow Y^{2m+1}\xrightarrow{\pi}X^{2m}$ over a
base of even dimension $2m=n-1$. We assume that $X$ is oriented
and equipped with a metric $g^{TX}$ and a spin structure. Let $TS^{1}=T^{V}Y\subset TY$
denote the subbundle of $TY$ consisting of the vertical tangent vectors.
Let $T^{H}Y\subset TY$ be another subbundle corresponding to a connection
on the fibration and hence giving an invariant splitting 
\begin{align}
TY= & T^{V}Y\oplus T^{H}Y\label{eq:tangent bundle splitting}
\end{align}
of the tangent bundle. The projection $\pi$ gives an identification
$T^{H}Y=\pi^{*}TX$. Consider the trivializing section of $TS^{1}$
given by $e_{y}=\frac{\partial}{\partial t}(e^{it}.y)|_{t=0}\in T_{y}S^{1}$,
the infinitesimal generator of the $S^{1}$ action. Let $g^{TS^{1}}$
be the metric on $TS^{1}$ such that $\left\Vert e\right\Vert _{g^{TS^{1}}}=1$.
We now consider the adiabatic family of metrics 
\begin{equation}
g_{\varepsilon}^{TY}=g^{TS^{1}}\oplus\varepsilon^{-1}\pi^{*}g^{TX}\label{eq:def. adiabatic metrics}
\end{equation}
on $Y$ as in \cite{Bismut-Cheeger}. 

Let $\nabla^{TY,\varepsilon},\nabla^{TX}$ denote the Levi-Civita
connections of $g_{\varepsilon}^{TY},g^{TX}$ respectively. The connection
$\nabla^{TY,\varepsilon}$ need not preserve the splitting \prettyref{eq:tangent bundle splitting}.
Let $p^{TS^{1}},p^{H}$ denote the projections of $TY$ onto $TS^{1},T^{H}Y$
respectively. Define a connection on $TS^{1}$ via $\nabla^{TS^{1}}=p^{TS^{1}}\nabla^{TY,\varepsilon}$.
As shown in section $4$ of \cite{Bismut-Cheeger}, the connection
$\nabla^{TS^{1}}$ is independent of $\varepsilon$. In the case of
circle bundles this is easily checked by showing that $e$ is $\nabla^{TS^{1}}$
-parallel via 
\begin{eqnarray}
\left\langle \nabla_{U}^{TS^{1}}e,e\right\rangle  & = & \left\langle \nabla_{U}^{TY,\varepsilon}e,e\right\rangle \nonumber \\
 & = & \frac{1}{2}U\left\langle e,e\right\rangle =0,\quad\forall U\in TY.\label{eq:e is parallel wrt horizontal connection}
\end{eqnarray}
Define the second connection $\nabla$ on $TY$ to be $\nabla=\nabla^{TS^{1}}\oplus\pi^{*}\nabla^{TX}$.
The connection $\nabla$ does preserve the splitting of $TX$ but
need not be torsion free. Let $T$ denote the torsion tensor of $\nabla$
and define the difference tensor 
\[
S^{\varepsilon}=\nabla^{TY,\varepsilon}-\nabla.
\]
Since $\nabla^{TY,\varepsilon}$ is torsion free, for tangent vectors
$U,V,W\in TY$ we have 
\[
S^{\varepsilon}(U)V-S^{\varepsilon}(V)U=-T(U,V).
\]
Since both $\nabla^{TY,\varepsilon},\nabla$ are compatible with $g_{\varepsilon}^{TY}$,
we also have
\[
\left\langle S^{\varepsilon}(U)V,W\right\rangle +\left\langle V,S^{\varepsilon}(U)W\right\rangle =0
\]
 where $\left\langle \right\rangle =g_{\varepsilon}^{TY}$. The last
two equations give
\begin{equation}
\left\langle S^{\varepsilon}(U)V,W\right\rangle =\frac{1}{2}\left[\left\langle T(V,W),U\right\rangle -\left\langle T(W,U),V\right\rangle -\left\langle T(U,V),W\right\rangle \right].\label{eq:Difference tensor in terms of Torsion tensor}
\end{equation}
Next we let $g^{TY}=g_{1}^{TY}$, $\nabla^{TY}=\nabla^{TY,1}$ and
$S=\nabla^{TY}-\nabla$ be the metric, Levi-Civita connection and
the difference tensor respectively when the adiabatic parameter $\varepsilon=1$
is set to one. It is clear from equations \prettyref{eq:def. adiabatic metrics}
and \prettyref{eq:Difference tensor in terms of Torsion tensor} that
\begin{eqnarray}
p^{H}S^{\varepsilon} & = & \varepsilon p^{H}S,\label{eq:Projections difference tensor 1}\\
p^{TS^{1}}S^{\varepsilon} & = & p^{TS^{1}}S.\label{eq:Projections difference tensor 2}
\end{eqnarray}
 Let $f_{i}$ be a locally defined orthonormal frame of vector fields
on the base $X$ and $\tilde{f_{i}}$ their lifts to $Y$. The torsion
tensor $T$ can be computed in the following cases to be 
\begin{enumerate}
\item $T(e,e)=0$ as $T$ is antisymmetric,
\item $T(e,\tilde{f})=\nabla_{e}\tilde{f}-\nabla_{\tilde{f}}e-[e,\tilde{f}]=-[e,\tilde{f}]=0$
for $\tilde{f}\in T^{H}Y$, as $\tilde{f}$ is $S^{1}$ invariant,
\item $T(\tilde{f}_{1},\tilde{f}_{2})=R(f_{1},f_{2})$ the curvature of
the $S^{1}$ connection, as $\nabla^{TX}$ is torsion free.
\end{enumerate}
Following the above computation of the torsion tensor, \prettyref{eq:Difference tensor in terms of Torsion tensor}
now clearly implies 
\begin{equation}
S(e)e=0.\label{eq:Difference tensor annihilates e}
\end{equation}
Define $e^{*}$ to be the one form which annihilates $T^{H}M$ and
$e^{*}(e)=1$. We then compute $de^{*}\left(\tilde{f}_{1},\tilde{f}_{2}\right)=-\left\langle \left[\tilde{f}_{1},\tilde{f}_{2}\right],e\right\rangle =R\left(f_{1},f_{2}\right)$
is the curvature of the $S^{1}$ connection.

\subsection{Splitting of the Dirac operator}

The spin structure on $TX$ can be pulled back to one on $T^{H}Y$.
Combined with the trivial spin structure on $TS^{1}$ this gives a
spin structure on $TY$. If $S_{\pm}^{TX}$denote the bundles of positive
and negative spinors on $X$, we have the identification $S^{TY}=\pi^{*}(S_{+}^{TX}\oplus S_{-}^{TX})$.
This in turn gives 
\[
C^{\infty}(Y,S^{TY})=C^{\infty}(X;(S_{+}^{TX}\oplus S_{-}^{TX})\otimes C^{\infty}(Y_{x})).
\]
 We now decompose 
\begin{equation}
C^{\infty}(Y_{x})=\bigoplus_{k\in\mathbb{Z}}E_{k}\label{eq:Fourier decomposition along fibre}
\end{equation}
according to the eigenspaces of $e$. Each $E_{k}$ corresponds to
the eigenvalue $-ik$ and gives a line bundle over $X$. Let $\mathcal{L}\rightarrow X$
be the Hermitian line bundle over $X$ corresponding to the standard
representation of $S^{1}$. Note that we may reconstruct $Y$ as the
space of unit elements in $\mathcal{L}$. We now also have the identification
$E_{k}=\mathcal{L}^{\otimes k}$. For any point $y_{x}\in Y_{x}\subset\mathcal{L}_{x}$,
this identification maps $y_{x}^{\otimes k}$ to $\{f(y_{x}e^{i\theta})=e^{-ik\theta}\}$
and we extend it by linearity. Hence we have the decomposition of
the space of spinors on $Y$ into 
\begin{equation}
C^{\infty}(Y,S^{TY})=\bigoplus_{k\in\mathbb{Z}}C^{\infty}(X;(S_{+}^{TX}\oplus S_{-}^{TX})\otimes\mathcal{L}^{\otimes k}).\label{eq:Decomposition of spinors}
\end{equation}
Finally, we twist the spin bundle $S^{TY}\otimes\mathbb{C}=S^{TY}$
by the trivial Hermitian line bundle but equipped with the family
of unitary connections $A_{r}=d+ire^{*}$. 

We now consider how the family of coupled Dirac operators $D_{A_{r},\varepsilon}$
decomposes with respect to \prettyref{eq:Decomposition of spinors}.
Let $\nabla^{S^{TY},\varepsilon},\tilde{\nabla}$ denote the lifts
of $\nabla^{TY,\varepsilon},\nabla$ to the spin bundle. Let $c^{\varepsilon}$
stand for the Clifford multiplication associated to the metric $g_{\varepsilon}^{TY}$.
We let $e_{i}=\varepsilon^{1/2}\tilde{f_{i}}$, where $\tilde{f_{i}}$
denote the locally defined horizontal lifts introduced earlier, and
also adopt the notation $e=e_{0}$. Using \prettyref{eq:Difference tensor in terms of Torsion tensor}
and the computation of the torsion tensor done in the previous subsection,
we now compute 
\begin{eqnarray}
D_{A_{r},\varepsilon} & = & \sum_{i=0}^{2m}c^{\varepsilon}(e_{i})\nabla_{e_{i}}^{S^{TY},\varepsilon}+irc^{\varepsilon}(e)\nonumber \\
 & = & c^{\varepsilon}(e)\tilde{\nabla}_{e}+\sum_{i=1}^{2m}c^{\varepsilon}(e_{i})\tilde{\nabla}_{e_{i}}+\frac{1}{4}\sum_{ijk}\left\langle S^{\varepsilon}(e_{i})e_{j},e_{k}\right\rangle c^{\varepsilon}(e_{i})c^{\varepsilon}(e_{j})c^{\varepsilon}(e_{k})+irc^{\varepsilon}(e)\nonumber \\
 & = & \bigoplus_{k}\begin{bmatrix}k-\frac{i\varepsilon}{4}c(R)-r & \varepsilon^{1/2}D_{-}^{B,k}\\
\varepsilon^{1/2}D_{+}^{B,k} & -k+\frac{i\varepsilon}{4}c(R)+r
\end{bmatrix}.\label{eq:Dirac op. splitting}
\end{eqnarray}
Here $D_{\pm}^{X,k}$ denotes the coupled Dirac operators acting on
sections of $S_{\pm}^{TX}\otimes\mathcal{L}^{\otimes k}$ and $c(R)=\sum_{i<j}R(f_{i},f_{j})c(f_{i})c(f_{j})$
denotes the Clifford multiplication by the curvature tensor $R$ on
the base $X$.

\subsection{\label{sub:The-Kahler-case}The Kahler case}

We now specialize to the case when $X$ is a complex manifold, with
complex structure $J$. We further assume $\mathcal{L}$ to be a positive,
holomorphic, Hermitian line bundle on $X$ . The curvature $R$ of
the associated holomorphic connection is now a $(1,1)$ form. Positivity
of $\mathcal{L}$ here means that $\frac{1}{2}R=\omega$ is a Kahler
form on $X$ (i.e. $\omega(.,J.)=g^{TX}(.,.)$ is a metric). A spin
structure on $X$ corresponds to a holomorphic, Hermitian square root
$\mathcal{K}$ of the canonical line bundle $\mathcal{K}^{\otimes2}=K_{X}$
(cf . \cite{Hitchin}). The corresponding bundles of positive and
negative spinors are $\Lambda^{even}T^{0,1*}\otimes\mathcal{K}$ and
$\Lambda^{odd}T^{0,1*}\otimes\mathcal{K}$ respectively while the
spin Dirac operator is $\sqrt{2}(\bar{\partial}_{\mathcal{K}}+\bar{\partial}_{\mathcal{K}}^{*})$
with $\bar{\partial}_{\mathcal{K}}$ being the holomorphic derivative
on $\Lambda^{*}T^{0,1*}\otimes\mathcal{K}$. Similarly the twisted
Dirac operator acting on sections of $\Lambda^{*}T^{0,1*}\otimes\mathcal{K}\otimes\mathcal{L}^{\otimes k}$is
given by $\sqrt{2}(\bar{\partial}_{\mathcal{K}\otimes\mathcal{L}^{\otimes k}}+\bar{\partial}_{\mathcal{K}\otimes\mathcal{L}^{\otimes k}}^{*})$.
Denote the holomorphic derivative $\bar{\partial}_{\mathcal{K}\otimes\mathcal{L}^{\otimes k}}$
on $\mathcal{K}\otimes\mathcal{L}^{\otimes k}$ by the shorthand $\bar{\partial}_{k}$
and let $\Delta_{\bar{\partial}_{k}}=\bar{\partial}_{k}\bar{\partial}_{k}^{*}+\bar{\partial}_{k}^{*}\bar{\partial}_{k}$
denote the Hodge Laplacian. Clifford multiplication by the Kahler
form is $c(\omega)=i(2N-m)$ where $N$ is the number operator which
acts as multiplication by $p$ on $\Lambda^{p}T^{0,1*}$. Hence the
formula \prettyref{eq:Dirac op. splitting} for the Dirac operator
is seen to specialize to 
\begin{equation}
D_{A_{r},\varepsilon}=\bigoplus_{k}\begin{bmatrix}k+\varepsilon(N-\frac{m}{2})-r & (2\varepsilon)^{1/2}(\bar{\partial}_{k}+\bar{\partial}_{k}^{*})\\
(2\varepsilon)^{1/2}(\bar{\partial}_{k}+\bar{\partial}_{k}^{*}) & -k-\varepsilon(N-\frac{m}{2})+r
\end{bmatrix}.\label{eq:Kahler Dirac op. splitting}
\end{equation}

Denote by $\mathcal{A}^{0,p}\left(\mathcal{L}^{\otimes k}\right)$
the space of $\mathcal{L}^{\otimes k}$-valued $(0,p)$ forms. Let
\[
\mathcal{A}^{0,p}\left(\mathcal{L}^{\otimes k}\right)=\bigoplus_{\mu\geq0}E_{\mu}^{p,k}
\]
be the spectral decomposition of $\Delta_{\bar{\partial}_{k}}$ where
$E_{\mu}^{p,k}$ denotes the eigenspace with eigenvalue $\frac{1}{2}\mu^{2}$.
From \prettyref{eq:Kahler Dirac op. splitting} it is clear that $\left[D_{A_{r}},\Delta_{\bar{\partial}_{k}}\right]=0$
and hence the Dirac operator preserves the eigenspaces $\bigoplus_{p}E_{\mu}^{p,k}$
of $\Delta_{\bar{\partial}_{k}}$ for each $\mu$. Let $\textrm{dim}\, E_{\mu}^{p,k}=e_{\mu}^{p,k}$
and define $d_{\mu}^{p,k}=e_{\mu}^{p,k}-e_{\mu}^{p-1,k}+\ldots+(-1)^{p}e_{\mu}^{0,k}$. 
\begin{lem}
\label{lem:Lemma to compute spectrum}For each positive $\mu>0$ we
have $d_{\mu}^{p,k}\geq0$. Furthermore there exists a collection
of $\bar{\partial}_{k}^{*}$-closed $p$ forms $\left\{ \omega_{j}^{p}\right\} _{j=1}^{d_{\mu}^{p,k}},$
such that $\left\{ \omega_{j}^{p}\right\} _{j=1}^{d_{\mu}^{p,k}}\cup\left\{ \bar{\partial}_{k}\omega_{j}^{p-1}\right\} _{j=1}^{d_{\mu}^{p-1,k}}$is
a basis of $E_{\mu}^{p,k}$.\end{lem}
\begin{proof}
We proceed by induction on $p$. Clearly $d_{\mu}^{0,k}=e_{\mu}^{0,k}\geq0$.
We take $\left\{ \omega_{j}^{0}\right\} _{j=1}^{d_{\mu}^{0,k}}$to
be any basis of $E_{\mu}^{0,k}$. Now assume that $\omega_{j}^{p}$
have been defined. Since they are $\bar{\partial}_{k}^{*}$-closed
we have 
\begin{eqnarray*}
\bar{\partial}_{k}^{*}\bar{\partial}_{k}\omega_{j}^{p} & =\Delta_{\bar{\partial_{k}}}\omega_{j}^{p} & =\frac{1}{2}\mu^{2}\omega_{j}^{p}\\
\Delta_{\bar{\partial_{k}}}\bar{\partial}_{k}\omega_{j}^{p} & =\bar{\partial}_{k}\bar{\partial}_{k}^{*}\bar{\partial}_{k}\omega_{j}^{p} & =\frac{1}{2}\mu^{2}\bar{\partial}_{k}\omega_{j}^{p}
\end{eqnarray*}
Hence the collection of forms $\left\{ \bar{\partial}_{k}\omega_{j}^{p}\right\} _{j=1}^{d_{\mu}^{p,k}}$
is linearly independent inside $E_{\mu}^{p+1,k}$. This implies $d_{\mu}^{p+1,k}=e_{\mu}^{p,k}-d_{\mu}^{p,k}\geq0$.
We choose $\left\{ \omega_{j}^{p+1}\right\} _{j=1}^{d_{\mu}^{p+1,k}}$
to be any basis for the orthogonal complement of the span the space
of $\left\{ \bar{\partial}_{k}\omega_{j}^{p}\right\} _{j=1}^{d_{\mu}^{p,k}}$
inside $E_{\mu}^{p+1,k}$. That each $\omega_{j}^{p+1}$ is $\bar{\partial}_{k}^{*}$-closed
follows from
\begin{eqnarray*}
\left\langle \bar{\partial}_{k}^{*}\omega_{j}^{p+1},\omega_{j'}^{p}\right\rangle  & =\left\langle \omega_{j}^{p+1},\bar{\partial}_{k}\omega_{j'}^{p}\right\rangle  & =0\\
\left\langle \bar{\partial}_{k}^{*}\omega_{j}^{p+1},\bar{\partial}_{k}\omega_{j'}^{p-1}\right\rangle  & =\left\langle \omega_{j}^{p+1},\bar{\partial}_{k}^{2}\omega_{j'}^{p-1}\right\rangle  & =0.
\end{eqnarray*}
Finally, $\left\{ \bar{\partial}_{k}\omega_{j}^{2m-1}\right\} _{j=1}^{d_{2m-1}}$
span $E_{\mu}^{2m,k}$ since the Dirac operator $\sqrt{2}(\bar{\partial}_{k}+\bar{\partial}_{k}^{*})$
is an isomorphism between $E_{\mu}^{even,k}$ and $E_{\mu}^{odd,k}$
for $\mu>0$.
\end{proof}
Following this lemma, we see using \prettyref{eq:Kahler Dirac op. splitting}
that the Dirac operator preserves the two dimensional subspaces of
spanned by $\left\{ \omega_{j}^{p},\frac{\sqrt{2}}{\mu}\bar{\partial}_{k}\omega_{j}^{p}\right\} $,
for each $0\leq j\leq d_{\mu}^{p,k}$. Furthermore its restriction
to this two dimensional subspace is given by the matrix
\[
\begin{bmatrix}(-1)^{p}(k+\varepsilon(p-\frac{m}{2})-r) & \mu\varepsilon^{1/2}\\
\mu\varepsilon^{1/2} & (-1)^{p+1}(k+\varepsilon(p+1-\frac{m}{2})-r)
\end{bmatrix}.
\]
The eigenvalues of the above matrix are computed to be
\[
\lambda=\frac{(-1)^{p+1}\varepsilon\pm\sqrt{(2k+\varepsilon(2p-m)-2r+1)^{2}+4\mu^{2}\varepsilon}}{2}.
\]
From \prettyref{eq:Kahler Dirac op. splitting} we also see that each
of $E_{0}^{p,k}$ is an eigenspace of the $D_{A_{r},\varepsilon}$
with eigenvalue $\lambda=(-1)^{p}(k+\varepsilon(p-\frac{m}{2})-r)$.
From Hodge theory, we have $E_{0}^{p,k}=H^{p}(X,\mathcal{K}\otimes\mathcal{L}^{\otimes k})$
and we denote 
\begin{align*}
h^{p,k}:= & e_{0}^{p,k}=\dim\, H^{p}(X,\mathcal{K}\otimes\mathcal{L}^{\otimes k}).
\end{align*}
 To sum up we have the following computation.
\begin{prop}
\label{prop:Computation of Dirac Spectrum}The eigenvalues of the
Dirac operator are given by the two types 
\begin{enumerate}
\item Type 1:
\[
\lambda=(-1)^{p}(k+\varepsilon(p-\frac{m}{2})-r),\;0\leq p\leq m,k\in\mathbb{Z}
\]
with multiplicity \textup{$h^{p,k}=dim\, H^{p}(X,\mathcal{K}\otimes\mathcal{L}^{\otimes k})$.}
\item Type 2:
\[
\lambda=\frac{(-1)^{p+1}\varepsilon\pm\sqrt{(2k+\varepsilon(2p-m)-2r+1)^{2}+4\mu^{2}\varepsilon}}{2},\;0\leq p\leq m,k\in\mathbb{Z}
\]
and $\frac{1}{2}\mu^{2}$ is a positive eigenvalue of $\Delta_{\bar{\partial_{k}}}^{p}$.
The multiplicity of $\lambda$ is $d_{\mu}^{p,k}=e_{\mu}^{p,k}-e_{\mu}^{p-1,k}+\ldots+(-1)^{p}e_{\mu}^{0,k}$
where $e_{\mu}^{p,k}$ is the multiplicity of $\frac{1}{2}\mu^{2}$.
\end{enumerate}
\end{prop}
The above proposition will allow us to compute an asymptotic formula
for the eta invariant $\bar{\eta}^{r,\varepsilon}=\bar{\eta}(D_{A_{r},\varepsilon})$
as $r\rightarrow\infty$, for each value of the adiabatic parameter
$\varepsilon$. To this end we first compute the spectral flow function
$\textrm{sf}\left\{ D_{A_{s},\varepsilon}\right\} _{0\leq s\leq r}$.
Note that the eigenvalues of type 2 do not change sign when the corresponding
positive eigenvalue of the Hodge Laplacian satisfies
\begin{equation}
\frac{1}{2}\mu^{2}>\frac{\varepsilon}{8}.\label{eq:Type 2 eigvenvalue condition}
\end{equation}
Let $K_{X}^{*}$ be the anticanonical bundle of $X$ and define $R_{ij}^{\mathcal{K}\otimes\mathcal{L}^{\otimes k}\otimes K_{X}^{*}}$
to be the curvature of the connection on $\mathcal{K}\otimes\mathcal{L}^{\otimes k}\otimes K_{X}^{*}$.
Let $\bar{dz_{i}}$ be an orthonormal basis of $T^{0,1^{*}}X$, with
$dz_{i}$ the dual basis of $T^{1,0{}^{*}}X$, and define $\lambda(R^{\mathcal{K}\otimes\mathcal{L}^{\otimes k}\otimes K_{X}^{*}})=\sum_{ij}R_{ij}^{\mathcal{K}\otimes\mathcal{L}^{\otimes k}\otimes K_{X}^{*}}\bar{dz_{i}}\wedge i_{dz_{j}}$.
Then the Bochner-Kodaira-Nakano formula (cf. \cite{Berline-Getzler-Vergne}
Proposition 3.71) asserts the existence of a positive operator $\Delta_{\bar{\partial_{k}}}^{p,0}$
such that
\begin{align*}
\Delta_{\bar{\partial_{k}}}^{p}= & \Delta_{\bar{\partial_{k}}}^{p,0}+\lambda(R^{\mathcal{K}\otimes\mathcal{L}^{\otimes k}\otimes K_{X}^{*}}).
\end{align*}

Now let $\alpha$ be a normalized eigenvector of $\Delta_{\bar{\partial_{k}}}^{p}$with
eigenvalue $\frac{1}{2}\mu^{2}$. We then compute 
\begin{eqnarray}
\frac{1}{2}\mu^{2}=\left\langle \Delta_{\bar{\partial_{k}}}^{p}\alpha,\alpha\right\rangle  & = & \left\langle \Delta_{\bar{\partial_{k}}}^{p,0}\alpha+\lambda(R^{\mathcal{K}\otimes\mathcal{L}^{\otimes k}\otimes K_{X}^{*}})\alpha,\alpha\right\rangle \nonumber \\
 & \geq & k\left\langle \lambda(R^{\mathcal{L}})\alpha,\alpha\right\rangle +\left\langle \lambda(R^{\mathcal{K}\otimes K_{X}^{*}})\alpha,\alpha\right\rangle >\frac{\varepsilon}{8},\label{eq:Lower bound First positive eigenvalue}
\end{eqnarray}
for $p>0$ and $k\gg0$ sufficiently large, via the positivity of
$\mathcal{L}$. In the case where $p=0$, since $\alpha$ is an eigenvector
of $\Delta_{\bar{\partial_{k}}}^{0}$ with positive eigenvalue and
$[\bar{\partial_{k}},\Delta_{\bar{\partial_{k}}}]=0$, we have that
$\bar{\partial_{k}}s$ is nonzero eigenvector of $\Delta_{\bar{\partial_{k}}}^{1}$
with the same eigenvalue. Hence \prettyref{eq:Type 2 eigvenvalue condition}
holds for each positive eigenvalue of $\Delta_{\bar{\partial_{k}}}^{p}$
for all $p$ and $k\gg0$ sufficiently positive. Finally using duality
we have that any positive eigenvalue of $\Delta_{\bar{\partial_{k}}}^{p}$
also obeys \prettyref{eq:Type 2 eigvenvalue condition} for all $p$
and $k\ll0$ sufficiently negative. Hence we see that there are at
most finitely many eigenvalues of $D_{A_{s},\varepsilon}$ of type
2 that change sign as $s$ varies from $0$ to $r$. The contribution
of the eigenvalues of type 1 to spectral flow is computed easily and
we have 
\[
\textrm{sf}\left\{ D_{A_{s},\varepsilon}\right\} _{0\leq s\leq r}=\sum_{0\leq k+\varepsilon(p-\frac{m}{2})\leq r}(-1)^{p+1}\, h^{p.k}+O(1).
\]
By the Kodaira vanishing theorem we have $h^{p,k}=0$ for $p>0$ and
$k\gg0$ sufficiently large. We hence have 
\[
h^{0,k}=\chi(X,\mathcal{K}\otimes\mathcal{L}^{\otimes k})=\int_{X}\textrm{ch}(\mathcal{K}\otimes\mathcal{L}^{\otimes k})\textrm{td}(X)
\]
 by the Hirzebruch-Riemann-Roch theorem. On using $\textrm{ch}(\mathcal{K}\otimes\mathcal{L}^{\otimes k})=\exp\left\{ kc_{1}(\mathcal{L})\right\} \exp\left\{ c_{1}(\mathcal{K})\right\} $
we have
\[
\textrm{sf}\left\{ D_{A_{s},\varepsilon}\right\} _{0\leq s\leq r}=-\sum_{k=0}^{\left[r+\frac{\varepsilon m}{2}\right]}\int_{X}\exp\left\{ kc_{1}(\mathcal{L})\right\} \exp\left\{ c_{1}(\mathcal{K})\right\} \textrm{td}(X)+O(1).
\]
Finally using the Atiyah-Patodi-Singer index theorem we have
\[
\bar{\eta}^{r,\varepsilon}=\bar{\eta}^{0,\varepsilon}+2\left\{ \textrm{sf}\left\{ D_{A_{s},\varepsilon}\right\} _{0\leq s\leq r}+\int_{0}^{r}ds\int_{X}\exp\left\{ sc_{1}(\mathcal{L})\right\} \exp\left\{ c_{1}(\mathcal{K})\right\} td(X)\right\} .
\]
Hence we have proved
\begin{thm}
\label{thm:Eta invrian asmp S1 bundle}The eta invariant $\bar{\eta}^{r,\varepsilon}$
satisfies the asymptotics
\begin{align}
\bar{\eta}^{r,\varepsilon}= & \sum_{a=0}^{m}\left\{ \left(\frac{r^{a+1}}{(a+1)!}-\sum_{k=1}^{\left[r+\frac{\varepsilon m}{2}\right]}\frac{k^{a}}{a!}\right)\int_{X}c_{1}(\mathcal{L})^{a}\left[\textrm{ch}(\mathcal{K})\textrm{td}(X)\right]^{m-a}\right\} +O(1)\label{eq:Asymptotic formula eta S1 bundle}
\end{align}
as $r\rightarrow\infty$.
\end{thm}
The above result shows that the eta invariant is discontinuous of
$O(r^{\frac{n-1}{2}})$ in this example.

\subsection{Computation of the eta invariant}

Although theorem \prettyref{thm:Eta invrian asmp S1 bundle} establishes
an asymptotic formula for the eta invariant it does not provide an
explicit computation for the eta invariant because of the $O(1)$
term in \prettyref{eq:Asymptotic formula eta S1 bundle}. In this
subsection we give an explicit computation of the eta invariant $\bar{\eta}^{r,\varepsilon}$,
assuming the value of the adiabatic parameter $\varepsilon$ to be
sufficiently small, using the adiabatic limit technique. We shall
first compute the $\hat{\eta}$-form of Bismut and Cheeger \cite{Bismut-Cheeger}
for circle bundles. This computation is similar to the one done by
Zhang in \cite{Zhang} with the only difference being the presence
of an extra coupling.

\subsubsection{The $\hat{\eta}$-form}

Let $S^{TS^{1}}$ denote the spin bundle of $TS^{1}$ and $\nabla^{S^{TS^{1}}}$be
the lift of $\nabla^{TS^{1}}$to $S^{TS^{1}}$. Let $\nabla^{S^{TS^{1}},r}=\nabla^{S^{TS^{1}}}\otimes1+1\otimes(d+ire^{*})$
denote the tensor product connection on $S^{TS^{1}}\otimes\mathbb{C}$.
Consider the infinite dimensional bundles over $X$ given by $H_{x}=C^{\infty}(Y_{x},S^{TS^{1}}\otimes\mathbb{C})$
and $G_{x}=C^{\infty}(Y_{x},TY)$. The connection $\nabla^{S^{TS^{1}},r}$
naturally lifts to a connection $\tilde{\nabla}^{S^{TS^{1}},r}$on
$H$. The torsion tensor $T$ may be considered as an element of $T\in\Omega^{2}(X,G)$
and we may define Clifford multiplication by the torsion tensor as
an element of $c(T)\in\Omega^{2}(\mathrm{End}(H))$. The fibrewise
Dirac operator can also be defined as an element of $D^{S^{1},r}=c(e)\nabla_{e}^{S^{TS^{1}},r}\in\mathrm{End}(H)$.
The Bismut superconnection on $H$ is defined via 
\[
A_{u}=\tilde{\nabla}^{S^{TS^{1}},r}+u^{1/2}D^{S^{1},r}-(4u)^{-1/2}c(T).
\]
The $\hat{\eta}$-form is the even form on $X$ defined by 
\begin{equation}
\hat{\eta}=\frac{1}{\sqrt{\pi}}\int_{0}^{\infty}\textrm{tr}^{\textrm{even}}\left[\left(D^{S^{1},r}+(4u)^{-1}c(T)\right)e^{-A_{u}^{2}}\right]\frac{du}{2u^{1/2}}.\label{eq:def eta form}
\end{equation}
Let $z$ be an auxiliary Grassman variable. Since the scalar curvature
of the circle is zero equations \prettyref{eq:e is parallel wrt horizontal connection}
and \prettyref{eq:Difference tensor annihilates e} simplify (4.68)-(4.70)
of \cite{Bismut-Cheeger} to give
\begin{equation}
-u\left(\nabla_{e}^{S^{TS^{1}},r}+\frac{R}{4u}+z\frac{c(e)}{2u^{1/2}}\right)^{2}+irR=A_{u}^{2}-z\left(u^{1/2}D^{S^{1},r}+(4u)^{-1/2}c(T)\right),\label{eq:curvature identity from bismut cheeger}
\end{equation}
where both sides are considered as operators on $\Omega^{*}(X,H)$.
If $\textrm{tr}^{z}(a+zb)=\textrm{tr}(b)$ then the above curvature
identity gives
\begin{multline*}
\textrm{tr}^{\textrm{even}}\left[\left(D^{S^{1},r}+(4u)^{-1}c(T)\right)e^{-A_{u}^{2}}\right]\\
=u^{-1/2}\textrm{tr}^{z}\left[\exp\left\{ u\left(\nabla_{e}^{S^{TS^{1}},r}+\frac{R}{4u}+z\frac{c(e)}{2u^{1/2}}\right)^{2}\right\} \right]\exp\left\{ -irR\right\} .
\end{multline*}

Next, the trivialization given by $e$ for $TS^{1}$ induces once
for $S^{TS^{1}}$. This allows us to identify each fiber
\begin{eqnarray*}
H_{x} & = & C^{\infty}(Y_{x},S^{TS^{1}}\otimes\mathbb{C})\\
 & = & C^{\infty}(Y_{x})=\bigoplus_{k}E_{k}
\end{eqnarray*}
by \prettyref{eq:Fourier decomposition along fibre}. Using $c(e)=-i$,
each $E_{k}$ is seen to be an eigenspace of $D^{S^{1},r}$ with eigenvalue
$-k+r$. Hence $D^{S^{1},r}$ is invertible for $r\notin\mathbb{Z}$.
We then have 
\begin{eqnarray}
 &  & u^{-1/2}\textrm{tr}^{z}\left[\exp\left\{ u\left(\nabla_{e}^{S^{TS^{1}},r}+\frac{R}{4u}+z\frac{c(e)}{2u^{1/2}}\right)^{2}\right\} \right]\\
 & = & u^{-1/2}\textrm{tr}^{z}\left[\sum_{k=-\infty}^{\infty}\exp\left\{ u\left(ik+ir+\frac{R}{4u}-\frac{iz}{2u^{1/2}}\right)^{2}\right\} \right]\\
 & = & -i\sum_{k=-\infty}^{\infty}\left(\frac{\pi}{u}\right)^{3/2}e^{2\pi ki(\frac{R}{4ui}+r)}\cdot k\cdot e^{-\frac{k^{2}\pi^{2}}{u}}
\end{eqnarray}
 where the last equality follows from a Poisson summation formula.
Hence

\begin{eqnarray}
 &  & \frac{1}{\sqrt{\pi}}\int_{0}^{\infty}u^{-1/2}\textrm{tr}^{z}\left[\exp\left\{ u\left(\nabla_{e}^{S^{TS^{1}},r}+\frac{R}{4u}+z\frac{c(e)}{2u^{1/2}}\right)^{2}\right\} \right]\frac{du}{2u^{1/2}}\\
 & = & \pi\int_{0}^{\infty}\sum_{k=1}^{\infty}k\cdot\sin\left(\frac{2\pi k}{u}\cdot\frac{R}{4i}+2\pi kr\right)e^{-\frac{k^{2}\pi^{2}}{u}}\frac{du}{u^{2}}\\
 & = & \sum_{k=1}^{\infty}\frac{k\pi\sin(2\pi kr)+\cos(2\pi kr)\cdot\frac{R}{2i}}{k^{2}\pi^{2}+\left(\frac{R}{2i}\right)^{2}}\\
 & = & \begin{cases}
\frac{1}{2}\left[\frac{\exp\left((1-2\{r\})\frac{R}{2i}\right)}{\sinh\left(\frac{R}{2i}\right)}-\frac{1}{R/2i}\right] & \textrm{if }r\notin\mathbb{Z}\\
\frac{1}{2}\left[\frac{\frac{R}{2i}-\tanh\left(\frac{R}{2i}\right)}{\frac{R}{2i}\tanh\left(\frac{R}{2i}\right)}\right] & \textrm{if }r\in\mathbb{Z}.
\end{cases}\label{eq:def special function for eta form}
\end{eqnarray}
 Here $\{r\}$ denotes the fractional part of $r\notin\mathbb{Z}$.
Let us denote the expression on line \prettyref{eq:def special function for eta form}
by $f\left(\frac{R}{2i},r\right)$. We note that this is a periodic
function in $r$ of period $1$. Hence we finally have that the eta
form is given by

\begin{equation}
\hat{\eta}=f\left(\frac{R}{2i},r\right)\exp\{-irR\}.
\end{equation}

\subsubsection{Adiabatic limit of the eta invariant}

Following the computation of the $\hat{\eta}$-form from the previous
section we now compute the adiabatic limit of the eta invariant. First
assume that the fibrewise Dirac operator $D^{S^{1},r}$ is invertible,
or $r\notin\mathbb{Z}$. The adiabatic limit of the eta invariant
is then given by proposition 4.95 of \cite{Bismut-Cheeger} to be
\begin{eqnarray*}
\lim_{\varepsilon\rightarrow0}\eta^{r,\varepsilon} & = & \frac{1}{(2\pi i)^{m}}\int_{X}\hat{A}(iR^{TX})\hat{\eta}\\
 & = & \int_{X}\hat{A}(X)\, f\left(\frac{c_{1}(\mathcal{L})}{2},r\right)\,\exp\left\{ rc_{1}(\mathcal{L})\right\} .
\end{eqnarray*}

In the case where $r=k\in\mathbb{Z}$ we have that $\textrm{ker}\left(D^{S^{1},r}\right)=E_{k}=\mathcal{L}^{\otimes k}$
forms a vector bundle over the base $X$. Furthermore, it is clear
from \prettyref{prop:Computation of Dirac Spectrum} that the dimension
of the kernel of $D_{A_{r},\varepsilon}$, for $\varepsilon$ small,
is given by 
\begin{align}
\textrm{dim ker}\left(D_{A_{r},\varepsilon}\right)= & \begin{cases}
h^{\frac{m}{2},k} & \textrm{if }m\quad\textrm{even}\\
0 & \textrm{if }m\quad\textrm{odd.}
\end{cases}\label{eq:dim ker in the limit}
\end{align}

Hence by Theorem 0.1 of \cite{Dai} the adiabatic limit of the eta
invariant exists in this case and is given by 
\begin{eqnarray}
\lim_{\varepsilon\rightarrow0}\eta^{r,\varepsilon} & = & \frac{1}{(2\pi i)^{m}}\int_{X}\hat{A}(iR^{TX})\hat{\eta}+\eta(\bar{\partial}_{k}+\bar{\partial}_{k}^{*})+\lim_{\varepsilon\rightarrow0}\sum_{\lambda_{0},\lambda_{1}=0}\textrm{sgn}(\lambda_{\varepsilon}).\label{eq:Dai Adiabatic limit formula}
\end{eqnarray}
Here the third term denotes a sum over the eigenvalues of $D_{A_{r},\varepsilon}$,
which vanish to $O(\varepsilon)$ as $\varepsilon\rightarrow0$, with
the convention $\textrm{sgn}(0)=0$. To compute this term note that
the eigenvalues of type 2 in \prettyref{prop:Computation of Dirac Spectrum}
do not vanish as $\varepsilon\rightarrow0$ for $r\in\mathbb{Z}$.
The eigenvalues of type 1 on the other hand vanish for $k=r$ and
the third is seen to be 
\[
\lim_{\varepsilon\rightarrow0}\sum_{\lambda_{0},\lambda_{1}=0}\textrm{sgn}(\lambda_{\varepsilon})=\sum_{p>\frac{m}{2}}\left(-1\right)^{p}h^{p,k}-\sum_{p<\frac{m}{2}}\left(-1\right)^{p}h^{p,k}
\]
Since the spectrum of $\sqrt{2}(\bar{\partial}_{k}+\bar{\partial}_{k}^{*})$
is symmetric, we have $\eta(\bar{\partial}_{k}+\bar{\partial}_{k}^{*})=0$.
Denoting by $c=c_{1}(\mathcal{L})$, we now sum up the calculation
of the adiabatic limit of the reduced eta invariant in all cases to
be
\[
\lim_{\varepsilon\rightarrow0}\bar{\eta}^{r,\varepsilon}=\begin{cases}
\frac{1}{2}\int_{X}\hat{A}(X)\,\left[\frac{\exp\left((1-2\{r\})\frac{c}{2}\right)}{\sinh\left(\frac{c}{2}\right)}-\frac{1}{c/2}\right]\,\exp\left\{ rc\right\} , & \textrm{if }r\notin\mathbb{Z},\\
\frac{1}{2}\Biggl\{\int_{X}\hat{A}(X)\,\left[\frac{\frac{c}{2}-\tanh\left(\frac{c}{2}\right)}{\frac{c}{2}\tanh\left(\frac{c}{2}\right)}\right]\,\exp\left\{ kc\right\} +h^{\frac{m}{2},k}\\
\qquad+\sum_{p>\frac{m}{2}}\left(-1\right)^{p}h^{p,k}-\sum_{p<\frac{m}{2}}\left(-1\right)^{p}h^{p,k}\Biggr\}, & \textrm{if }r=k\in\mathbb{Z},\; m\;\textrm{even},\\
\frac{1}{2}\Biggl\{\int_{X}\hat{A}(X)\,\left[\frac{\frac{c}{2}-\tanh\left(\frac{c}{2}\right)}{\frac{c}{2}\tanh\left(\frac{c}{2}\right)}\right]\,\exp\left\{ kc\right\} \\
\qquad+\sum_{p>\frac{m}{2}}\left(-1\right)^{p}h^{p,k}-\sum_{p<\frac{m}{2}}\left(-1\right)^{p}h^{p,k}\Biggr\}, & \textrm{if }r=k\in\mathbb{Z},\; m\;\textrm{odd}.
\end{cases}
\]

\subsubsection{Spectral flow function}

To proceed with the computation of the eta invariant $\bar{\eta}^{r,\varepsilon}$
we attempt to compute the spectral flow function $\textrm{sf}\left\{ D_{A_{r},\delta}\right\} _{0\leq\delta\leq\varepsilon}$.
Let $\textrm{Spec}^{+}(A)$ denote the positive spectrum of an operator
$A$ and define 
\[
M=\inf_{k,p}\left\{ \frac{1}{2}\mu^{2}\in\textrm{Spec}^{+}\left(\Delta_{\bar{\partial_{k}}}^{p}\right)\right\} .
\]
By the arguments in subsection \prettyref{sub:The-Kahler-case} we
have $M>0$. Furthermore, if we choose the adiabatic parameter small
enough so that $\frac{\varepsilon}{8}<M$, the eigenvalues of type
2 in \prettyref{prop:Computation of Dirac Spectrum} do not contribute
to spectral flow. The spectral flow from the eigenvalues of type 1
is easily computed to give
\begin{multline*}
\textrm{sf}\left\{ D_{A_{r},\delta}\right\} _{0\leq\delta\leq\varepsilon}=\sum_{p>\frac{m}{2},\textrm{even}}\:\sum_{k=\left\lceil r-\varepsilon\left(p-\frac{m}{2}\right)\right\rceil }^{\left\lceil r\right\rceil -1}\, h^{p,k}-\sum_{p>\frac{m}{2},\textrm{odd}}\:\sum_{k=\left\lfloor r-\varepsilon\left(p-\frac{m}{2}\right)\right\rfloor +1}^{\left\lfloor r\right\rfloor }\, h^{p,k}\\
-\sum_{p<\frac{m}{2},\textrm{even}}\:\sum_{k=\left\lceil r\right\rceil }^{\left\lceil r-\varepsilon\left(p-\frac{m}{2}\right)\right\rceil -1}\, h^{p,k}+\sum_{p<\frac{m}{2},\textrm{odd}}\:\sum_{k=\left\lfloor r\right\rfloor +1}^{\left\lfloor r-\varepsilon\left(p-\frac{m}{2}\right)\right\rfloor }\, h^{p,k}.
\end{multline*}
Here $\left\lfloor x\right\rfloor ,\left\lceil x\right\rceil $ stand
for the floor and ceiling functions of $x$ respectively.

\subsubsection{The transgression form}

Next let $\left\{ \nabla^{\delta}\right\} _{0\leq\delta\leq\varepsilon}$
be any family of connections on $TY$ such that $\nabla^{0}=\nabla^{TY,0},\nabla^{\varepsilon}=\nabla^{TY,\varepsilon}$.
This family determines a connection $\nabla^{TZ}$ on the tangent
bundle $TZ$ of $Z=Y\times[0,\varepsilon]_{\delta}$ via
\[
\nabla^{TZ}=d\delta\wedge\frac{\partial}{\partial\delta}+\nabla^{\delta}.
\]
 Let $R^{TZ}$ be the curvature of $\nabla^{TZ}$. By the Atiyah-Patodi-Singer
index theorem we have 
\begin{equation}
\bar{\eta}^{r,\varepsilon}-\lim_{\varepsilon\rightarrow0}\bar{\eta}^{r,\varepsilon}=2\left\{ \textrm{sf}\left\{ D_{A_{r},\delta}\right\} _{0\leq\delta\leq\varepsilon}+\frac{1}{\left(2\pi i\right)^{m+1}}\int_{Z}\,\hat{A}(R^{TZ})\right\} .\label{eq:Transgression eta invariants}
\end{equation}
Note that this in particular implies that the integral term above
is independent of the chosen family of connections. Here we shall
compute the form $\hat{A}(R^{TZ})$. We choose the natural family
of connections 
\[
\nabla^{\delta}=\nabla^{TY,\delta}=\nabla+\delta p^{H}S+p^{TS^{1}}S
\]
by \prettyref{eq:Projections difference tensor 1},\prettyref{eq:Projections difference tensor 2}.
Denoting $p^{H}S=S^{H},p^{TS^{1}}S=S^{V}$ by shorthands we have
\begin{eqnarray*}
R^{TZ} & = & \left(d\delta\wedge\frac{\partial}{\partial\delta}+\nabla+\delta S^{H}+S^{V}\right)^{2}\\
 & = & d\delta\wedge S^{H}+R^{TX}+\nabla S^{V}+S^{V}\wedge S^{V}\\
 &  & +\delta\left(\nabla S^{H}+S^{H}\wedge S^{V}+S^{V}\wedge S^{H}\right)+\delta^{2}S^{H}\wedge S^{H}.
\end{eqnarray*}

Next we compute using \prettyref{eq:Difference tensor in terms of Torsion tensor}
\begin{eqnarray*}
S^{H}(e)e=0, &  & S^{H}(e)f=Jf,\\
S^{H}(f)e=Jf, &  & S^{H}(f_{1})f_{2}=0,
\end{eqnarray*}
as well as
\begin{eqnarray*}
S^{V}(e)e=0, &  & S^{V}(e)f=0,\\
S^{V}(f)e=0, &  & S^{V}(f_{1})f_{2}=-\omega(f_{1},f_{2})e.
\end{eqnarray*}
where $f_{1},f_{2}\in T^{H}Y=\pi^{*}TX$. We may hence write $S^{H}=e^{*}\otimes J+\alpha_{1}$
where $e^{*}\otimes J,\alpha_{1}\in\Omega^{1}(Y;\mathrm{End}(TY))$
with the only nonzero combination of $\alpha_{1}$ being 
\begin{equation}
\alpha_{1}(f)e=Jf.\label{eq:Alpha 1 formula}
\end{equation}
Following this we may compute in symplectic geodesic coordinates to
obtain $\nabla S^{V}=0$, while computing in holomorphic geodesic
coordinates yields $\nabla S^{H}=de^{*}\otimes J=2\omega\otimes J\in\Omega^{2}(X;\mathrm{End}(TX)).$
Computing further we find $S^{V}\wedge S^{V}=0$. Another computation
gives $S^{H}\wedge S^{V}=\Omega\in\Omega^{2}(X;\mathrm{End}(TX))\subset\Omega^{2}(Y;\mathrm{End}(TY))$
is given by
\begin{equation}
\Omega(f_{1},f_{2})f=\omega(f_{1},f)Jf_{2}-\omega(f_{2},f)Jf_{1}.\label{eq:Novel tensor}
\end{equation}
Also $S^{V}\wedge S^{H}=e^{*}\wedge\alpha_{2}$, where $\alpha_{2}\in\Omega^{1}(Y;\mathrm{End}(TY))$
whose only nonzero combination is 
\begin{equation}
\alpha_{2}(f_{1})f=g(f_{1},f)e.\label{eq:Alpha 2 formula}
\end{equation}
And $S^{H}\wedge S^{H}=e^{*}\wedge\alpha_{3}$ where $\alpha_{3}\in\Omega^{1}(Y;\mathrm{End}(TY))$
whose only nonzero combination is 
\begin{equation}
\alpha_{3}(f_{1})e=-f_{1}.\label{eq:Alpha 3 formula}
\end{equation}

We hence have 
\begin{multline}
R^{TZ}=d\delta\wedge e^{*}\otimes J+d\delta\wedge\alpha_{1}+R^{TX}+2\delta\omega\otimes J\\
+\delta\Omega+\delta e^{*}\wedge\alpha_{2}+\delta^{2}e^{*}\wedge\alpha_{3}\label{eq:Terms in curvature of cylinder}
\end{multline}
and we wish to compute
\begin{eqnarray*}
\hat{A}(R^{TZ}) & = & \exp\left\{ \textrm{tr}p(R^{TZ})\right\} ,\qquad\textrm{where}\\
p(z) & = & \frac{1}{2}\log\left(\frac{z/2}{\sinh\left(z/2\right)}\right).
\end{eqnarray*}
Since $p(z)$ is an even function in $z$ vanishing at zero, it has
a power series 
\begin{align}
p(z)= & p_{2}z^{2}+p_{4}z^{4}+\ldots.\label{eq:power series p}
\end{align}
We shall begin our computation of $\textrm{tr}p(R^{TZ})$ with the
following lemma.
\begin{lem}
\label{lem:Tensor identities}We have the tensor identities
\begin{enumerate}
\item .
\begin{align}
\Omega\wedge\Omega= & 0\label{eq:OwO}
\end{align}

\item .
\begin{align}
\Omega\wedge R^{TX}= & R^{TX}\wedge\Omega=0,\label{eq:OwR=00003DRwO=00003D0}
\end{align}

\item .
\begin{align}
R^{TX}\wedge\alpha_{1}= & 0,\label{eq:RwALH}
\end{align}

\item .
\begin{align}
\textrm{tr}\left[\left(\Omega J\right)^{\wedge k}\right]= & -2^{k}\omega^{\wedge k}.\label{eq:tr (OJ)^k}
\end{align}

\end{enumerate}
\end{lem}
\begin{proof}
Let $f_{i}$ denote an orthonormal basis of $TX$ at a point. We also
denote by $S_{k}$ the group of permutations of $\left\{ 1,2,\ldots,k\right\} .$
\begin{enumerate}
\item We compute 
\begin{align}
 & \Omega\wedge\Omega(f_{1},f_{2},f_{3},f_{4})f\nonumber \\
= & \frac{1}{4}\sum_{\sigma\in S_{4}}\textrm{sgn}\left(\sigma\right)\Omega\left(f_{\sigma(1)},f_{\sigma(2)}\right)\Omega\left(f_{\sigma(3)},f_{\sigma(4)}\right)f\nonumber \\
= & \frac{1}{4}\sum_{\sigma\in S_{4}}\textrm{sgn}\left(\sigma\right)\Biggl\{\omega\left(f_{\sigma(1)},Jf_{\sigma(4)}\right)\omega\left(f_{\sigma(3)},f\right)Jf_{\sigma(2)}-\omega\left(f_{\sigma(1)},Jf_{\sigma(4)}\right)\omega\left(f_{\sigma(3)},f\right)Jf_{\sigma(2)}\nonumber \\
 & \qquad\qquad-\omega\left(f_{\sigma(1)},Jf_{\sigma(4)}\right)\omega\left(f_{\sigma(3)},f\right)Jf_{\sigma(2)}+\omega\left(f_{\sigma(1)},Jf_{\sigma(4)}\right)\omega\left(f_{\sigma(3)},f\right)Jf_{\sigma(2)}\Biggr\}\label{eq:OwO comp.}\\
= & 0,\nonumber 
\end{align}
since each of the four terms in \prettyref{eq:OwO comp.} contains
an expression of the type $\omega\left(f_{i},Jf_{j}\right)=g^{TX}\left(f_{i},f_{j}\right)=0.$
\item We have $R^{TX}\in\Omega^{2}\left(\mathfrak{so}\left(TX\right)\right)$
and $[R^{TX},J]=0$ since the complex structure $J$ is parallel.
This gives the identity  $\omega\left(f_{1},R^{TX}\left(f_{3},f_{4}\right)f_{4}\right)=g^{TX}\left(R^{TX}\left(f_{3},f_{4}\right)f_{1},Jf_{4}\right)$.
We then compute 
\begin{align*}
 & \Omega\wedge R^{TX}(f_{1},f_{2},f_{3},f_{4})f\\
= & \frac{1}{4}\sum_{\sigma\in S_{4}}\textrm{sgn}\left(\sigma\right)\Omega\left(f_{\sigma(1)},f_{\sigma(2)}\right)R^{TX}\left(f_{\sigma(3)},f_{\sigma(4)}\right)f\\
= & \frac{1}{4}\sum_{\sigma\in S_{4}}\textrm{sgn}\left(\sigma\right)\Biggl\{\omega\left(f_{\sigma(1)},R^{TX}\left(f_{\sigma(3)},f_{\sigma(4)}\right)f\right)Jf_{\sigma(2)}\\
 & \qquad\qquad\qquad\qquad\qquad-\omega\left(f_{\sigma(2)},R^{TX}\left(f_{\sigma(3)},f_{\sigma(4)}\right)f\right)Jf_{\sigma(1)}\Biggr\}\\
= & \frac{1}{4}\sum_{\sigma\in S_{4}}\textrm{sgn}\left(\sigma\right)\Biggl\{ g^{TX}\left(R^{TX}\left(f_{\sigma(3)},f_{\sigma(4)}\right)f_{\sigma(1)},Jf\right)Jf_{\sigma(2)}\\
 & \qquad\qquad\qquad\qquad\qquad-g^{TX}\left(R^{TX}\left(f_{\sigma(3)},f_{\sigma(4)}\right)f_{\sigma(2)},Jf\right)Jf_{\sigma(1)}\Biggr\}\\
= & 0,
\end{align*}
by Bianchi's identity. The computation $R^{TX}\wedge\Omega=0$ is
similar.
\item We compute
\begin{align*}
 & R^{TX}\wedge\alpha_{1}\left(f_{1},f_{2},f_{3}\right)e\\
= & \frac{1}{2}\sum_{\sigma\in S_{3}}\textrm{sgn}\left(\sigma\right)R^{TX}\left(f_{\sigma(1)},f_{\sigma(2)}\right)\alpha_{1}\left(f_{\sigma(3)}\right)e\\
= & \frac{1}{2}\sum_{\sigma\in S_{3}}\textrm{sgn}\left(\sigma\right)R^{TX}\left(f_{\sigma(1)},f_{\sigma(2)}\right)Jf_{\sigma(3)}\\
= & \frac{J}{2}\left\{ \sum_{\sigma\in S_{3}}\textrm{sgn}\left(\sigma\right)R^{TX}\left(f_{\sigma(1)},f_{\sigma(2)}\right)f_{\sigma(3)}\right\} \\
= & 0,
\end{align*}
by Bianchi's identity.
\item For each $1\leq l\leq k,$ define the transposition $\tau_{l}=\left(2l-1\quad2l\right)\in S_{2k}$.
Given a subset $S\subset\left\{ 1,\ldots,k\right\} $, define the
permutation 
\[
\tau_{S}=\prod_{l\in S}\tau_{l}
\]
in $S_{2k}$. Using the definition \prettyref{eq:Novel tensor} of
$\Omega$, we now compute
\begin{align*}
 & \textrm{tr}\left[\left(\Omega J\right)^{\wedge k}\left(f_{1},\ldots,f_{2k}\right)\right]\\
= & \frac{1}{2^{k}}\sum_{\sigma\in S_{2k}}\textrm{sgn}\left(\sigma\right)\textrm{tr}\left[\prod_{l=1}^{k}\Omega\left(f_{\sigma\left(2l-1\right)},f_{\sigma\left(2l\right)}\right)J\right]\\
= & \frac{1}{2^{k}}\sum_{\sigma\in S_{2k}}\textrm{sgn}\left(\sigma\right)\left\{ \sum_{i=1}^{2k}g^{TX}\left(\left[\prod_{l=1}^{k}\Omega\left(f_{\sigma\left(2l-1\right)},f_{\sigma\left(2l\right)}\right)J\right]f_{i},f_{i}\right)\right\} \\
= & \frac{1}{2^{k}}\sum_{\sigma\in S_{2k}}\textrm{sgn}\left(\sigma\right)\left\{ \sum_{S\subset\left\{ 1,\ldots,k\right\} }\textrm{sgn}\left(\tau_{S}\right)\left[\prod_{l=1}^{k-1}\omega\left(f_{\tau_{S}\circ\sigma\left(2l\right)},f_{\tau_{S}\circ\sigma\left(2l+1\right)}\right)\right]\omega\left(f_{\tau_{S}\circ\sigma\left(2l\right)},f_{\tau_{S}\circ\sigma\left(1\right)}\right)\right\} \\
= & \sum_{S\subset\left\{ 1,\ldots,k\right\} }\left\{ \frac{1}{2^{k}}\sum_{\sigma\in S_{2k}}\textrm{sgn}\left(\tau_{S}\circ\sigma\right)\left[\prod_{l=1}^{k-1}\omega\left(f_{\tau_{S}\circ\sigma\left(2l\right)},f_{\tau_{S}\circ\sigma\left(2l+1\right)}\right)\right]\omega\left(f_{\tau_{S}\circ\sigma\left(2l\right)},f_{\tau_{S}\circ\sigma\left(1\right)}\right)\right\} \\
= & -\sum_{S\subset\left\{ 1,\ldots,k\right\} }\omega^{\wedge k}\left(f_{1},\ldots,f_{2k}\right)\\
= & -2^{k}\omega^{\wedge k}\left(f_{1},\ldots,f_{2k}\right).
\end{align*}

\end{enumerate}
\end{proof}
We now perform further computations. Note that since the complex structure
is parallel, the complexification of the Levi-Civita connection $\nabla^{TX}$
preserves the holomorphic and anti-holomorphic tangent spaces $TX^{1,0},TX^{0,1}$.
Let $\nabla^{TX^{1,0}},\nabla^{TX^{0,1}}$ be the restrictions of
$\nabla^{TX}$ to $TX^{1,0},TX^{0,1}$ and let $R^{TX^{1,0}},R^{TX^{0,1}}$
denote their respective curvatures. One then has 
\begin{equation}
\frac{1}{2}\textrm{tr}\left[\left(JR^{TX}\right)^{N}\right]=\textrm{tr}\left[\left(iR^{TX^{1,0}}\right)^{N}\right],\label{eq:tr(JR)^N=00003Dtr(iR)^N}
\end{equation}
where the right hand side is now the trace of a complex linear endomorphism. 

Before stating the next computation, define the sequence $\left\{ \epsilon_{i}\right\} _{i=2}^{\infty}$
of integers via 
\begin{align*}
\epsilon_{N}= & \begin{cases}
1 & \textrm{if }N=2,\\
0 & \textrm{if }N>2.
\end{cases}
\end{align*}

\begin{prop}
Let $N\geq2$ be an even integer. The following identities hold
\begin{enumerate}
\item .
\begin{align}
\textrm{tr}\left[\left(R^{TX}+2\delta\omega\otimes J+\delta\Omega\right)^{N}\right]= & 2\textrm{tr}\left[\left(R^{TX^{1,0}}+2i\delta\omega\right)^{N}\right]+2\left(2i\delta\omega\right)^{N}.\label{eq:tr(R+w+O)^N}
\end{align}

\item .
\begin{align}
\textrm{tr}\left[J\left(R^{TX}+2\delta\omega\otimes J+\delta\Omega\right)^{N-1}\right]= & 2\textrm{tr}\left[i\left(R^{TX^{1,0}}+2i\delta\omega\right)^{N-1}\right]+2i\left(2i\delta\omega\right)^{N-1}+2\epsilon_{N}\omega.\label{eq:tr(R+w+O)^N-1}
\end{align}

\item .
\begin{align}
\textrm{tr}\left[\Omega J\left(R^{TX}+2\delta\omega\otimes J+\delta\Omega\right)^{N-2}\right]= & -2\epsilon_{N}\omega\label{eq:trOJ(R+w+O)^N-2}
\end{align}

\end{enumerate}
\end{prop}
\begin{proof}
(1). The expansion of $\left(R^{TX}+2\delta\omega\otimes J+\delta\Omega\right)^{N}$consists
of monomials in the three tensors $R^{TX},2\delta\omega\otimes J$
and $\delta\Omega$. Using $\left[R^{TX},\omega\otimes J\right]=0$
and \prettyref{eq:OwR=00003DRwO=00003D0} we see that a monomial containing
both $R^{TX}$ as well as $\Omega$ is necessarily zero. Hence 
\begin{align}
\left(R^{TX}+2\delta\omega\otimes J+\delta\Omega\right)^{N}= & \left(R^{TX}+2\delta\omega\otimes J\right)^{N}\label{eq:w & w/o Omega}\\
 & +\left(2\delta\omega\otimes J+\delta\Omega\right)^{N}-\left(2\delta\omega\otimes J\right)^{N}.\nonumber 
\end{align}
The trace of the first summand on the right hand side of \prettyref{eq:w & w/o Omega-1}
is easily computed using \prettyref{eq:tr(JR)^N=00003Dtr(iR)^N} to
be 
\begin{align*}
\textrm{tr}\left[\left(R^{TX}+2\delta\omega\otimes J\right)^{N}\right]= & 2\textrm{tr}\left[\left(R^{TX^{1,0}}+2i\delta\omega\right)^{N}\right].
\end{align*}
Next, for each $a=\left(a_{1},\ldots,a_{k+1}\right)\in\mathbb{N}_{0}^{k+1}$
we denote $|a|=\sum_{i=1}^{k+1}a_{i}.$ Then the sum of the last two
terms in \prettyref{eq:w & w/o Omega-1} is
\begin{multline}
\left(2\delta\omega\otimes J+\delta\Omega\right)^{N}-\left(2\delta\omega\otimes J\right)^{N}=\\
\sum_{k>0}\left\{ \sum_{\begin{subarray}{l}
\;\; a\in\mathbb{N}_{0}^{k+1}\\
|a|=N-k
\end{subarray}}\left(2\delta\omega J\right)^{a_{1}}\delta\Omega\ldots\delta\Omega\left(2\delta\omega J\right)^{a_{k+1}}\right\} .\label{eq:exp terms w w/o Omega}
\end{multline}
 Using identity \prettyref{eq:OwO}, we see that the only non-zero
terms in the sum \prettyref{eq:exp terms w w/o Omega} are ones satisfying
the parity constraint 
\begin{equation}
a_{1}+a_{k+1},a_{2},\ldots,a_{k}\:\textrm{odd}.\label{eq:parity const.}
\end{equation}
Furthermore, using \prettyref{eq:tr (OJ)^k}, we may compute the trace
of each summand in \prettyref{eq:exp terms w w/o Omega} satisfying
\prettyref{eq:parity const.} to be
\begin{align*}
\textrm{tr}\left[\left(2\delta\omega J\right)^{a_{1}}\delta\Omega\ldots\delta\Omega\left(2\delta\omega J\right)^{a_{k+1}}\right]= & -\left(-1\right)^{k}\left(2i\delta\omega\right)^{N}.
\end{align*}
The number of $a\in\mathbb{N}_{0}^{k+1}$ with $|a|=N-k$ and satisfying
\prettyref{eq:parity const.} is easily computed to be $2\binom{\frac{N}{2}}{k}$.
We hence have 
\begin{align*}
 & \textrm{tr}\left[\left(2\delta\omega\otimes J+\delta\Omega\right)^{N}-\left(2\delta\omega\otimes J\right)^{N}\right]\\
= & -\left(2i\delta\omega\right)^{N}\left\{ \sum_{k>0}\left(-1\right)^{k}2\binom{\frac{N}{2}}{k}\right\} \\
= & 2\left(2i\delta\omega\right)^{N}.
\end{align*}

(2). The proof is almost identical to part 1. Again we see that a
monomial in the expansion of $J\left(R^{TX}+2\delta\omega\otimes J+\delta\Omega\right)^{N-1}$
cannot contain both $R^{TX}$ and $\Omega$. Hence 
\begin{align}
J\left(R^{TX}+2\delta\omega\otimes J+\delta\Omega\right)^{N-1}= & J\biggl[\left(R^{TX}+2\delta\omega\otimes J\right)^{N-1}\nonumber \\
 & +\left(2\delta\omega\otimes J+\delta\Omega\right)^{N-1}-\left(2\delta\omega\otimes J\right)^{N-1}\biggr].\label{eq:w & w/o Omega-1}
\end{align}
The trace of the first term on the right hand side above is again
easily computed using \prettyref{eq:tr(JR)^N=00003Dtr(iR)^N} to be
\begin{align*}
\textrm{tr}\left[J\left(R^{TX}+2\delta\omega\otimes J\right)^{N-1}\right]= & 2\textrm{tr}\left[i\left(R^{TX^{1,0}}+2i\delta\omega\right)^{N-1}\right].
\end{align*}
The sum of the last two terms in \prettyref{eq:w & w/o Omega-1} is
now 
\begin{multline}
J\left(2\delta\omega\otimes J+\delta\Omega\right)^{N-1}-J\left(2\delta\omega\otimes J\right)^{N-1}=\\
\sum_{k>0}\left\{ \sum_{\begin{subarray}{l}
\;\; a\in\mathbb{N}_{0}^{k+1}\\
|a|=N-k-1
\end{subarray}}J\left(2\delta\omega J\right)^{a_{1}}\delta\Omega\ldots\delta\Omega\left(2\delta\omega J\right)^{a_{k+1}}\right\} .\label{eq:exp terms with Omega 2}
\end{multline}
Using identity \prettyref{eq:OwO}, we see that the only non-zero
terms in the sum \prettyref{eq:exp terms with Omega 3} are ones satisfying
the parity constraint 
\begin{equation}
a_{1}+a_{k+1}\;\textrm{even},\quad a_{2},\ldots,a_{k}\:\textrm{odd}.\label{eq:parity const. 2}
\end{equation}
Furthermore, using \prettyref{eq:tr (OJ)^k}, we may compute the trace
of each summand in \prettyref{eq:exp terms w w/o Omega} satisfying
\prettyref{eq:parity const. 2} to be
\begin{align*}
\textrm{tr}\left[J\left(2\delta\omega J\right)^{a_{1}}\delta\Omega\ldots\delta\Omega\left(2\delta\omega J\right)^{a_{k+1}}\right]=- & i\left(-1\right)^{k}\left(2i\delta\omega\right)^{N-1}.
\end{align*}
The number of $a\in\mathbb{N}_{0}^{k+1}$ with $|a|=N-1-k$ and satisfying
\prettyref{eq:parity const. 2} is again computed to be $\binom{\frac{N}{2}}{k}+\binom{\frac{N}{2}-1}{k}$.
We hence have 
\begin{align*}
 & \textrm{tr}\left[J\left(2\delta\omega\otimes J+\delta\Omega\right)^{N-1}-J\left(2\delta\omega\otimes J\right)^{N-1}\right]\\
= & -i\left(2i\delta\omega\right)^{N-1}\left\{ \sum_{k>0}\left(-1\right)^{k}\left[\binom{\frac{N}{2}}{k}+\binom{\frac{N}{2}-1}{k}\right]\right\} \\
= & 2i\left(2i\delta\omega\right)^{N-1}+2\epsilon_{N}\omega.
\end{align*}

(3). Since $\Omega J\left(R^{TX}+2\delta\omega\otimes J+\delta\Omega\right)^{N-2}$
already contains $\Omega$, the identity \prettyref{eq:OwO} now implies
\begin{align}
 & \Omega J\left(R^{TX}+2\delta\omega\otimes J+\delta\Omega\right)^{N-2}\nonumber \\
= & \Omega J\left(2\delta\omega\otimes J+\delta\Omega\right)^{N-2}\nonumber \\
= & \Omega J\left(2\delta\omega J\right)^{N-2}+\sum_{k>0}\left\{ \sum_{\begin{subarray}{l}
\;\; a\in\mathbb{N}_{0}^{k+1}\\
|a|=N-k-2
\end{subarray}}\Omega J\left(2\delta\omega J\right)^{a_{1}}\delta\Omega\ldots\delta\Omega\left(2\delta\omega J\right)^{a_{k+1}}\right\} .\label{eq:exp terms with Omega 3}
\end{align}
Using identity \prettyref{eq:OwO}, we see that the only non-zero
terms in the sum \prettyref{eq:exp terms with Omega 3} are ones satisfying
the parity constraint 
\begin{equation}
a_{1}\;\textrm{even},\quad a_{2},\ldots,a_{k+1}\:\textrm{odd}.\label{eq:parity const. 3}
\end{equation}
Furthermore, using \prettyref{eq:tr (OJ)^k}, we may compute the trace
of each summand in \prettyref{eq:exp terms w w/o Omega} satisfying
\prettyref{eq:parity const. 3} to be
\begin{align*}
\textrm{tr}\left[\Omega J\left(2\delta\omega J\right)^{a_{1}}\delta\Omega\ldots\delta\Omega\left(2\delta\omega J\right)^{a_{k+1}}\right]=- & 2\omega\left(-1\right)^{k}\left(2i\delta\omega\right)^{N-2}.
\end{align*}
The number of $a\in\mathbb{N}_{0}^{k+1}$ with $|a|=N-2-k$ and satisfying
\prettyref{eq:parity const. 3} is again computed to be $\binom{\frac{N}{2}-1}{k}$.
We hence have 
\begin{align*}
 & \textrm{tr}\left[\Omega J\left(2\delta\omega\otimes J+\delta\Omega\right)^{N-1}-J\left(2\delta\omega\otimes J\right)^{N-1}\right]\\
= & -2\omega\left(2i\delta\omega\right)^{N-2}-2\left(2i\delta\omega\right)^{N-2}\left\{ \sum_{k>0}\left(-1\right)^{k}2\binom{\frac{N}{2}-1}{k}\right\} \\
= & -2\epsilon_{N}\omega.
\end{align*}
\end{proof}
\begin{prop}
Let $N\geq2$ be an even integer. The following identity holds
\begin{align}
\textrm{tr}\left(R^{TZ}\right)^{N}= & 2\textrm{tr}\left[\left(R^{TX^{1,0}}+2i\delta\omega\right)^{N}\right]+2\left(2i\delta\omega\right)^{N}\nonumber \\
 & +d\delta\wedge e^{*}\,\left\{ 2\textrm{tr}\left[iN\left(R^{TX^{1,0}}+i2\delta\omega\right)^{N-1}\right]+i2N\left(2i\delta\omega\right)^{N-1}\right\} .\label{eq:tr R^N}
\end{align}
\end{prop}
\begin{proof}
Clearly $\left(R^{TZ}\right)^{N}$ is a sum of monomials in the seven
tensors appearing on the right hand side of \prettyref{eq:Terms in curvature of cylinder}.
Due to the $d\delta$ and $e^{*}$ factors, a nonzero monomial appearing
in $\left(R^{TZ}\right)^{N}$ is of atmost degree two in the four
tensors $d\delta\wedge e^{*}\otimes J,d\delta\wedge\alpha_{1},\delta e^{*}\wedge\alpha_{2}$
and $\delta^{2}e^{*}\wedge\alpha_{3}$. Let $P_{i}$ denote the sum
of monomials of degree $i$ in these four tensors appearing in the
expansion of $\left(R^{TX}\right)^{N}$. Hence 
\begin{align}
\left(R^{TZ}\right)^{N}= & P_{0}+P_{1}+P_{2},\label{eq:R^N=00003DP0+P1+P2}
\end{align}
and we now compute the traces of $P_{0},P_{1}$ and $P_{2}$. 

\textbf{\textsc{Trace of $P_{0}$.}} It is clear that 
\begin{align}
\textrm{tr}P_{0}= & \textrm{tr}\left[\left(R^{TX}+2\delta\omega\otimes J+\delta\Omega\right)^{N}\right]\nonumber \\
= & 2\textrm{tr}\left[\left(R^{TX^{1,0}}+2i\delta\omega\right)^{N}\right]+2\left(2i\delta\omega\right)^{N}\label{eq:tr P0}
\end{align}
by \prettyref{eq:tr(R+w+O)^N}. 

\textbf{\textsc{Trace of $P_{1}$}}\textbf{\textit{.}} A monomial
in $P_{1}$ must contain exactly one occurrence of $d\delta\wedge\alpha_{1},\delta e^{*}\wedge\alpha_{2}$
or $\delta^{2}e^{*}\wedge\alpha_{3}$ and must not contain $d\delta\wedge e^{*}\otimes J$.
From the formulas \prettyref{eq:Alpha 1 formula}-\prettyref{eq:Alpha 3 formula}
for $\alpha_{1},\alpha_{2}$ and $\alpha_{3}$, it is clear that such
a monomial switches the $T^{H}Y$ and $T^{V}Y$ summands. Hence we
have
\begin{equation}
\textrm{tr}\, P_{1}=0.\label{eq:tr P1}
\end{equation}

\textbf{\textsc{Trace of $P_{2}$.}} A nonzero monomial in $P_{2}$
must contain a single appearance of $d\delta$ and $e^{*}$ each.
It can hence be of following three types. 

\textit{Type A.} This type of monomial contains a single appearance
of $d\delta\wedge e^{*}\otimes J$ and no appearances of $d\delta\wedge\alpha_{1},\delta e^{*}\wedge\alpha_{2}$
or $\delta^{2}e^{*}\wedge\alpha_{3}$. Let $P_{2}^{1}$ be the sum
of all monomials of this type appearing in $\left(R^{TX}\right)^{N}$.
Using the cyclicity of the trace we easily see that 
\begin{align}
\textrm{tr}P_{2}^{1}= & N\, d\delta\wedge e^{*}\,\textrm{tr}\left[J\left(R^{TX}+2\delta\omega\otimes J+\delta\Omega\right)^{N-1}\right]\nonumber \\
= & d\delta\wedge e^{*}\,\left\{ 2\textrm{tr}\left[iN\left(R^{TX^{1,0}}+2i\delta\omega\right)^{N-1}\right]+i2N\left(2i\delta\omega\right)^{N-1}+4\epsilon_{N}\omega\right\} ,\label{eq:tr P2^1}
\end{align}
by \prettyref{eq:tr(R+w+O)^N-1}.

\textit{Type B.} This type of monomial contains a single appearance
each of $d\delta\wedge\alpha_{1},\delta e^{*}\wedge\alpha_{2}$ and
no appearances of $d\delta\wedge e^{*}\otimes J$ or $\delta^{2}e^{*}\wedge\alpha_{3}$.
Let $P_{2}^{2}$ be the sum of all monomials of this type appearing
in $\left(R^{TX}\right)^{N}$. From the formulas \prettyref{eq:Alpha 1 formula}
and \prettyref{eq:Alpha 2 formula}, we note that $\alpha_{1}$ maps
$T^{V}Y$ into $T^{H}Y$ while $\alpha_{2}$ maps $T^{H}Y$ into $T^{V}Y$.
Hence in order to have a nonzero trace, a monomial of this type must
be of the form 
\begin{align*}
\delta e^{*}\wedge\alpha_{2}\wedge A\wedge d\delta\wedge\alpha_{1} & \quad\textrm{or}\\
B\wedge d\delta\wedge\alpha_{1}\wedge\delta e^{*}\wedge\alpha_{2}\wedge C,
\end{align*}
where $A,B$ and $C$ are some monomials in the tensors $R^{TX},2\delta\omega\otimes J$
and $\delta\Omega$. Thus we see that in a monomial of this type $d\delta\wedge\alpha_{1},\delta e^{*}\wedge\alpha_{2}$
appear consecutively after a cyclic permutation. Using the cyclicity
of the trace we now have
\begin{align*}
\textrm{tr}P_{2}^{2}= & N\,\textrm{tr}\left[d\delta\wedge\alpha_{1}\wedge\delta e^{*}\wedge\alpha_{2}\wedge\left(R^{TX}+2\delta\omega\otimes J+\delta\Omega\right)^{N-2}\right].
\end{align*}
The identity 
\begin{align*}
\alpha_{1}\wedge\alpha_{2}= & -\Omega J
\end{align*}
combined with \prettyref{eq:trOJ(R+w+O)^N-2} now gives 
\begin{align}
\textrm{tr}P_{2}^{2}= & -4\epsilon_{N}\delta\omega.\label{eq:tr P2^2}
\end{align}

\textit{Type C.} The third type of monomial contains one appearance
each of $d\delta\wedge\alpha_{1}$ and $\delta^{2}e^{*}\wedge\alpha_{3}$
and no appearances of $d\delta\wedge e^{*}\otimes J$ or $\delta e^{*}\wedge\alpha_{2}$.
However since $\alpha_{1}$and $\alpha_{3}$ both annihilate $T^{H}Y$
and map $T^{V}Y$ into $T^{H}Y$ such a monomial must necessarily
have trace zero. 

Adding \prettyref{eq:tr P2^1} and \prettyref{eq:tr P2^2} gives 
\begin{align}
\mbox{tr}P_{2}= & d\delta\wedge e^{*}\,\left\{ 2\textrm{tr}\left[iN\left(R^{TX^{1,0}}+2i\delta\omega\right)^{N-1}\right]+i2N\left(2i\delta\omega\right)^{N-1}\right\} .\label{eq:tr P2}
\end{align}
The proposition now follows from \prettyref{eq:R^N=00003DP0+P1+P2},
\prettyref{eq:tr P0}, \prettyref{eq:tr P1} and \prettyref{eq:tr P2}.
\end{proof}
Finally, substituting \prettyref{eq:tr R^N} into the power series
\prettyref{eq:power series p}, we now have 
\begin{align*}
\textrm{tr}\left\{ p\left(R^{TZ}\right)\right\} = & \Omega_{0}+d\delta\wedge e^{*}\wedge\Omega_{2},\quad\textrm{where}\\
\Omega_{0}= & 2\textrm{tr}\left[p\left(R^{TX^{1,0}}+2i\delta\omega\right)\right]+2p\left(2i\delta\omega\right)\\
\Omega_{2}= & 2\textrm{tr}\left[ip'\left(R^{TX^{1,0}}+i2\delta\omega\right)\right]+i2p'\left(2i\delta\omega\right).
\end{align*}
We may now calculate 
\begin{eqnarray*}
\hat{A}(R^{TZ}) & = & \int_{0}^{\varepsilon}\int_{Y}\exp\left\{ \Omega_{0}+d\delta\wedge e^{*}\wedge\Omega_{2}\right\} \\
 & = & \int_{0}^{\varepsilon}\int_{Y}d\delta\wedge e^{*}\wedge\Omega_{2}\exp\left\{ \Omega_{0}\right\} \\
 & = & \left(2\pi\right)\int_{0}^{\varepsilon}d\delta\int_{X}\Omega_{2}\exp\left\{ \Omega_{0}\right\} .
\end{eqnarray*}
In view of equation \prettyref{eq:Transgression eta invariants},
we now summarize the calculation of the eta invariant. 
\begin{thm}
\label{thm:Eta invariant explicit computation}The eta invariant $\bar{\eta}^{r,\varepsilon}$
for $\frac{\varepsilon}{8}<\inf_{k,p}\left\{ \frac{1}{2}\mu^{2}\in\textrm{Spec}^{+}\left(\Delta_{\bar{\partial_{k}}}^{p}\right)\right\} $
is given by 
\[
\bar{\eta}^{r,\varepsilon}=\lim_{\varepsilon\rightarrow0}\bar{\eta}^{r,\varepsilon}+\textrm{sf}\left\{ D_{A_{r},\delta}\right\} _{0\leq\delta\leq\varepsilon}+\frac{1}{\left(2\pi i\right)^{m+1}}\int_{Z}\,\hat{A}(R^{TZ})
\]
where the three terms above are given by
\begin{enumerate}
\item the adiabatic limit:
\[
\lim_{\varepsilon\rightarrow0}\bar{\eta}^{r,\varepsilon}=\begin{cases}
\frac{1}{2}\int_{X}\hat{A}(X)\,\left[\frac{\exp\left((1-2\{r\})\frac{c}{2}\right)}{\sinh\left(\frac{c}{2}\right)}-\frac{1}{c/2}\right]\,\exp\left\{ rc\right\} , & \textrm{if }r\notin\mathbb{Z},\\
\frac{1}{2}\Biggl\{\int_{X}\hat{A}(X)\,\left[\frac{\frac{c}{2}-\tanh\left(\frac{c}{2}\right)}{\frac{c}{2}\tanh\left(\frac{c}{2}\right)}\right]\,\exp\left\{ kc\right\} +h^{\frac{m}{2},k}\\
\qquad+\sum_{p>\frac{m}{2}}\left(-1\right)^{p}h^{p,k}-\sum_{p<\frac{m}{2}}\left(-1\right)^{p}h^{p,k}\Biggr\}, & \textrm{if }r=k\in\mathbb{Z},\; m\;\textrm{even},\\
\frac{1}{2}\Biggl\{\int_{X}\hat{A}(X)\,\left[\frac{\frac{c}{2}-\tanh\left(\frac{c}{2}\right)}{\frac{c}{2}\tanh\left(\frac{c}{2}\right)}\right]\,\exp\left\{ kc\right\} \\
\qquad+\sum_{p>\frac{m}{2}}\left(-1\right)^{p}h^{p,k}-\sum_{p<\frac{m}{2}}\left(-1\right)^{p}h^{p,k}\Biggr\}, & \textrm{if }r=k\in\mathbb{Z},\; m\;\textrm{odd},
\end{cases}
\]
with $c=c_{1}(\mathcal{L})$.
\item the spectral flow function:
\begin{multline*}
\textrm{sf}\left\{ D_{A_{r},\delta}\right\} _{0\leq\delta\leq\varepsilon}=\sum_{p>\frac{m}{2},\textrm{even}}\:\sum_{k=\left\lceil r-\varepsilon\left(p-\frac{m}{2}\right)\right\rceil }^{\left\lceil r\right\rceil -1}\, h^{p,k}-\sum_{p>\frac{m}{2},\textrm{odd}}\:\sum_{k=\left\lfloor r-\varepsilon\left(p-\frac{m}{2}\right)\right\rfloor +1}^{\left\lfloor r\right\rfloor }\, h^{p,k}\\
-\sum_{p<\frac{m}{2},\textrm{even}}\:\sum_{k=\left\lceil r\right\rceil }^{\left\lceil r-\varepsilon\left(p-\frac{m}{2}\right)\right\rceil -1}\, h^{p,k}+\sum_{p<\frac{m}{2},\textrm{odd}}\:\sum_{k=\left\lfloor r\right\rfloor +1}^{\left\lfloor r-\varepsilon\left(p-\frac{m}{2}\right)\right\rfloor }\, h^{p,k}.
\end{multline*}

\item the transgression form:
\begin{align*}
\int_{Z}\,\hat{A}(R^{TZ})= & \left(2\pi\right)\int_{0}^{\varepsilon}d\delta\int_{X}\Omega_{2}\exp\left\{ \Omega_{0}\right\} ,\quad\textrm{where}\\
\Omega_{0}= & 2\textrm{tr}\left[p\left(R^{TX^{1,0}}+2i\delta\omega\right)\right]+2p\left(2i\delta\omega\right),\\
\Omega_{2}= & 2\textrm{tr}\left[ip'\left(R^{TX^{1,0}}+i2\delta\omega\right)\right]+i2p'\left(2i\delta\omega\right)
\end{align*}
and $p(z)=\frac{1}{2}\log\left(\frac{z/2}{\sinh\left(z/2\right)}\right)$. 
\end{enumerate}
\end{thm}
We observe that the formula above expresses the eta invariant in purely
topological terms on the base. 

Finally, we show that our computation agrees with the one of Nicolaescu
from \cite{Nicolaescu-Eta} in dimension three. Consider the case
when $X$ is a oriented Riemann surface. We choose $g^{TX}$ a metric
of volume $\pi l$ where $l$ is a positive integer. Choose the complex
structure $J=-\star$ on $X$, where $\star$ denotes the Hodge star.
This gives a Kahler form $\omega$ satisfying $\int_{X}\omega=-\pi l$.
Let $\mathcal{L}\rightarrow X$ be a Hermitian line bundle of degree
$c_{1}\left(\mathcal{L}\right)=l$. This allows us to pick a connection
on $\mathcal{L}$ with curvature $R=2\omega$, which induces a holomorphic
structure on $\mathcal{L}$. We may now choose $Y$ to be the unit
circle bundle in $\mathcal{L}$ over $X$ equipped with the adiabatic
family of metrics \prettyref{eq:def. adiabatic metrics}. We now specialize
our formula for the eta invariant to compute $\bar{\eta}^{0,\varepsilon}$
in this case. Assuming the adiabatic parameter $\varepsilon$ to be
sufficiently small the spectral flow contribution in \prettyref{thm:Eta invariant explicit computation}
is seen to vanish. Setting $r=0$ the other terms in the formula are
easily computed to give 
\begin{align*}
\bar{\eta}^{0,\varepsilon}= & \frac{c}{12}-\frac{1}{2}\left(h^{1,0}+h^{0,0}\right)+\frac{\varepsilon^{2}l}{12}+\frac{\varepsilon}{12}\int_{X}\textrm{tr}\left[\frac{iR^{TX^{1,0}}}{2\pi}\right].
\end{align*}
Using Serre duality and Gauss-Bonnet we get 
\begin{align}
\bar{\eta}^{0,\varepsilon}= & \frac{c}{12}-h^{0,0}+\frac{\varepsilon^{2}l}{12}-\frac{\varepsilon\chi}{12},\label{eq:Nic. eta formula}
\end{align}
where $\chi$ is the Euler characteristic of the surface. However
the adiabatic metrics in \cite{Nicolaescu-Eta} were chosen to be
of the form $r^{2}g^{TS^{1}}\oplus\frac{1}{l}\pi^{*}g^{TX}$. This
amounts to a rescaling and hence the substitution $\varepsilon=r^{2}l$
in \prettyref{eq:Nic. eta formula}. Following this our formula is
seen to agree in this case with Theorem 2.4 proved, by two different
methods, in \cite{Nicolaescu-Eta}.

\appendix

\section{Estimates on Gaussian integrals \label{sec:Estimates-on-Gaussian}}

Here we prove some estimates on Gaussian integrals used in section
\ref{sec:Asymptotics-of-the}
\begin{lem}
There exist constants $C_{1},C_{2}$ and $C_{3}$ depending only on
the Riemannian manifold $(Y,g)$, such that for any $x,z\in Y$ and
$t,t'>0$ we have the following inequalities
\begin{enumerate}
\item .
\begin{equation}
\int_{Y}h_{t}(x,y)dy\leq C_{1},\label{eq:total integral Gaussian}
\end{equation}

\item . 
\begin{equation}
\int_{Y}h_{t}(x,y)h_{t'}(y,z)dy\leq C_{2}h_{4(t+t')}(x,z)\label{eq:Convolution estimate}
\end{equation}

\item and
\begin{equation}
\int_{0}^{t}s^{-\frac{1}{2}}ds\left(\int_{Y}dyh_{2\left(t-s\right)}(x,y)h_{2s}(y,z)\right)\leq C_{3}t^{\frac{1}{2}}h_{8t}(x,z).\label{eq:Duhamel convolution estimate}
\end{equation}

\end{enumerate}
\end{lem}
\begin{proof}
$\left(1\right).$ Consider the ball $B=\left\{ y|\rho(x,y)<i_{g}\right\} $
and split the integral \ref{eq:total integral Gaussian} into integrals
over \textbf{$B$} and its complement $B^{c}$. Introducing geodesic
coordinates, the integral over $B$ can be bounded from above by the
Euclidean integral $\int_{\mathbb{R}^{n}}\frac{e^{-\frac{|r|^{2}}{4t}}}{(4\pi t)^{\frac{n}{2}}}dr=1.$
For the integral over $B^{c}$, we use the inequality $\frac{e^{-\frac{\rho(x,y)^{2}}{4t}}}{(4\pi t)^{\frac{n}{2}}}\leq\frac{(n/2)!}{\pi^{n/2}i_{g}^{n}}$
to get $\int_{B^{c}}h_{t}(x,y)dy\leq\frac{(n/2)!}{\pi^{n/2}i_{g}^{n}}\textrm{vol}(Y)$.

$\left(2\right).$Without loss of generality assume that $t\leq t'.$
The triangle inequality gives the estimate $\frac{\rho(x,y)^{2}}{t}+\frac{\rho(y,z)^{2}}{t'}\geq\frac{\rho(x,z)^{2}}{2(t+t')}$.
Using this we may bound 
\begin{align*}
\int_{Y} & \frac{e^{-\frac{\rho(x,y)^{2}}{4t}}}{(4\pi t)^{\frac{n}{2}}}\frac{e^{-\frac{\rho(y,z)^{2}}{4t'}}}{(4\pi t')^{\frac{n}{2}}}dy\\
\leq & \frac{e^{-\frac{\rho(x,z)^{2}}{16(t+t')}}}{(4\pi t')^{\frac{n}{2}}}\left(\int_{Y}\frac{e^{-\frac{\rho(x,y)^{2}}{8t}}}{(4\pi t)^{\frac{n}{2}}}e^{-\frac{\rho(y,z)^{2}}{8t'}}dy\right).
\end{align*}
Then via $\frac{1}{t'}\leq\frac{2}{t+t'}$ and $e^{-\frac{\rho(y,z)^{2}}{8t'}}\leq1$
we may further bound this from above by $2^{n}h_{4(t+t')}(x,z)\left(\int_{Y}h_{2t}(x,y)dy\right).$
The estimate \ref{eq:Convolution estimate} now follows from \ref{eq:total integral Gaussian}.

$\left(3\right)$. First use \ref{eq:Convolution estimate} to estimate
\begin{align*}
 & \int_{0}^{t}s^{-\frac{1}{2}}\left(\int_{Y}dyh_{2\left(t-s\right)}(x,y)h_{2s}(y,z)\right)ds\\
\leq & C_{2}\int_{0}^{t}s^{-\frac{1}{2}}h_{8t}(x,z)ds=2C_{2}t^{\frac{1}{2}}h_{8t}(x,z).
\end{align*}

\end{proof}
\textbf{Acknowledgements.} The author would like to thank Prof. J.-M.
Bismut for several discussions and suggesting many key ideas presented
here. The author is also grateful to his thesis advisor Tom Mrowka
under whom this project was started. The author would like to thank
the referees for constructive comments and suggestions to improve
the paper.

\bibliographystyle{siam}
\addcontentsline{toc}{section}{\refname}\bibliography{biblio}

\end{document}